\documentclass[onefignum,onetabnum]{siamonline190516}

\usepackage{amsfonts,amssymb,units,bbm}
\usepackage{graphicx}
\usepackage{epstopdf}
\usepackage{algorithmic}
\usepackage{microtype}
    \usepackage[numbers,sort]{natbib}
\ifpdf
  \DeclareGraphicsExtensions{.eps,.pdf,.png,.jpg}
\else
  \DeclareGraphicsExtensions{.eps}
\fi

\newcommand{\W}{\mathcal W}
\newcommand{\R}{\mathbb R}
\newcommand{\E}{\mathbb E}
\newcommand{\N}{\mathbb N}
\newcommand{\Prob}{\mathrm{Prob}}
\newcommand{\Unif}{\mathrm{Unif}}
\newcommand{\Var}{\mathrm{Var}}
\newcommand{\Cov}{\mathrm{Cov}}

\newcommand{\st}{\mathrm{s.t.}}
\newcommand{\supp}{\mathrm{supp}}
\newcommand{\1}{\mathbbm1}

\renewcommand{\epsilon}{\varepsilon}

\usepackage{enumitem}
\setlist[enumerate]{leftmargin=.5in}
\setlist[itemize]{leftmargin=.5in}

\newsiamremark{remark}{Remark}
\newsiamremark{hypothesis}{Hypothesis}
\crefname{hypothesis}{Hypothesis}{Hypotheses}
\newsiamthm{claim}{Claim}
\newsiamthm{example}{Example}
\newsiamthm{conjecture}{Conjecture}
\headers{$k$-Variance: A Clustered Notion of Variance}{Solomon, Greenewald, and Nagaraja}

\title{$k$-Variance: A Clustered Notion of Variance\thanks{Originally posted in December 2020.
\funding{J.\ Solomon acknowledges the generous support of Army Research Office grants W911NF1710068 and W911NF2010168, of Air Force Office of Scientific Research award FA9550-19-1-031, of National Science Foundation grant IIS-1838071, from the CSAIL Systems that Learn program, from the MIT–IBM Watson AI Laboratory, from the Toyota–CSAIL Joint Research Center, from a gift from Adobe Systems, and from the Skoltech--MIT Next Generation Program.}}}

\author{
Justin Solomon\thanks{Department of Electrical Engineering and Computer Science, Massachusetts Institute of Technology, Cambridge, MA (\email{jsolomon@mit.edu}, \url{http://people.csail.mit.edu/jsolomon/}).}
\and
Kristjan Greenewald\thanks{MIT--IBM Watson AI Lab, Cambridge, MA (\email{Kristjan.H.Greenewald@ibm.com}, \url{https://kgreenewald.github.io/})}
\and
Haikady N.\ Nagaraja\thanks{Division of Biostatistics, The Ohio State University, Columbus, OH (\email{nagaraja.1@osu.edu}, \url{https://cph.osu.edu/people/hnagaraja})}
}

\usepackage{amsopn}

\begin{document}

\maketitle

\begin{abstract}
  We introduce $k$-variance, a generalization of variance built on the machinery of random bipartite matchings.  $K$-variance measures the expected cost of matching two sets of $k$ samples from a distribution to each other, capturing local rather than global information about a measure as $k$ increases; it is easily approximated stochastically using sampling and linear programming.  In addition to defining $k$-variance and proving its basic properties, we provide in-depth analysis of this quantity in several key cases, including one-dimensional measures, clustered measures, and measures concentrated on low-dimensional subsets of $\R^n$.  We conclude with experiments and open problems motivated by this new way to summarize distributional shape.
\end{abstract}



\section{Introduction}

A key task in statistics and data science is to describe the \emph{shape} of a dataset or distribution in a simple form.  The most basic means of summarizing distributions extract scalar measurements characterizing spread, normality, support, decay, and other aspects of distributional geometry.  Among these measurements, the simplest and most popular choice is \emph{variance}, which measures squared deviation of a random variable from its mean.

A scalar is unlikely to capture all relevant or interesting information about a distribution, and indeed variance is not sensitive to skew, asymmetry, and other structural properties.  A typical way to address this issue is to compute higher-order moments, which---if completely known---can often reconstruct a distribution.  While this solution works mathematically, each (standarized) moment measures the allotment of mass in a distribution relative to its mean, which is hard to interpret in the multi-modal or clustered cases.

In this paper, we introduce a generalization of variance we call $k$-variance, intended to address some of the issues above.  The basic idea of $k$-variance is to draw $2k$ samples from a distribution and to evaluate the transport cost of matching the first $k$ samples to the second $k$ samples. $K$-variance coincides with variance in the $k=1$ case. But, for larger values of $k$, samples get matched to closer counterparts in the distribution rather than between different modes, making $k$-variance a more localized measure of variance.

Our construction of $k$-variance seems to indicate that a tightly-clustered distribution about a few means might have high (1-)variance if those means are far apart, but that $k$-variance of such a distribution will decay rapidly in $k$ relative to that of a unimodal Gaussian.  Indeed, we will prove that this is the case---but only for measures embedded in dimensions $\gtrapprox 5$.  In lower dimensions, $k$-variance exhibits surprising---and somewhat counterintuitive---behavior, which we can capture in detail for one-dimensional $k$-variance using the theory of order statistics.

$K$-variance can be approximated using a simple randomized algorithm, wherein we draw $2k$ points and solve a $k\times k$ transportation problem; unsurprisingly, the accuracy of this easy-to-implement estimator can be improved by averaging over multiple trials.  We provide variance bounds demonstrating that $k$-variance requires fewer such trials as $k$ and/or the ambient dimension increases.

We conclude with some experiments demonstrating the behavior of $k$-variance as a measure of intra-mode variability, as well as a number of open problems motivated by our work.

\paragraph*{Contributions} We introduce a generalization of variance for probability measures on $\R^n$ we call ``$k$-variance,'' built on constructions from optimal transport.  Beyond introducing $k$-variance and its basic properties (\cref{sec:kvar}), we
\begin{itemize}[leftmargin=*]
    \item give alternative expressions and bounds for $k$-variance of probability measures over $\R$ (\cref{sec:1d});
    \item use results in empirical optimal transport to characterize $k$-variance of probability measures concentrated on low-dimensional sets (\cref{sec:lowdimensional}), higher-dimensional sets (\cref{sec:highdimensional}), and with cluster structure (\cref{sec:clustered});
    \item bound the variance of empirical estimators for $k$-variance in terms of sample size and dimension (\cref{sec:variance}); and
    \item provide numerical experiments to demonstrate behavior of $k$-variance and confirm our predicted theory (\cref{sec:experiments}).
\end{itemize}

\section{Related work} 

For the most part, we incorporate discussion of related work into the text below as it arises; our work principally uses results from the theory of optimal transport (cf.\ \cite{villani2003topics,santambrogio2015optimal,peyre2019computational}) and---in one dimension---from the theory of order statistics (cf.\ \cite{david2003order}).

Before commencing our technical discussion, however, we note that our work is built on recent advances in the theory of \emph{random Euclidean bipartite matchings}.  This theory seeks to characterize the cost of matching two independently-drawn $k$-samples of a measure to one another, where the cost of matching two points is proportional to the $p$-th power of Euclidean distance.  See \cite{dobric1995asymptotics,barthe2013combinatorial,de2002almost,dereich2013constructive,goldman2020convergence} and references therein for relevant mathematical theory, and see \cite{yukich2006probability,caracciolo2014scaling} for applications in other disciplines. While these works focus on bounding the transport cost in specific cases or connecting it to physical applications, here we show how the matching cost can be understood as a generalization of variance useful for characterizing the shape of a probability measure.

\section{Preliminaries} We begin with mathematical preliminaries to establish notation.

\subsection{Variance} Our work focuses on generalizing the \emph{variance} of a random variable $X$ drawn from a probability measure $\mu\in\Prob(\R^d)$, which is the expected squared deviation of that variable from its mean $\overline X:=\E[X]$:
\begin{equation}
    \Var(X):=\E_{X\sim\mu}[\|X-\overline X\|_2^2].
\end{equation}
A simple argument reveals an alternative formula for variance:
\begin{equation}\label{eq:variance_pairwise}
    \Var(X)=\frac12 \E_{X,Y\sim\mu}[\|X-Y\|_2^2].
\end{equation}

\subsection{Optimal transport}

Take $\mu,\nu\in\Prob(\R^d)$ to be two Radon probability measures.  Then, we can define the (squared) \emph{2-Wasserstein distance} between $\mu$ and $\nu$ via
\begin{equation}\label{eq:w2}
    \W_2^2(\mu,\nu):=\inf_{\pi\in\Pi(\mu,\nu)} \E_{(X,Y)\sim\pi}[\|X-Y\|_2^2],
\end{equation}
where $\Pi(\mu,\nu)\subseteq\Prob(\R^d\times\R^d)$ denotes the set of \emph{measure couplings} whose marginals are $\mu$ and $\nu$, resp. The Wasserstein distance is a basic object of study in analysis, statistics, machine learning, and related disciplines.  Intuitively, $\W_2(\mu,\nu)$ measures the amount of work it takes to displace $\mu$ onto $\nu$ as distributions of mass over $\R^d$, where the cost of moving a particle of mass from $x\in\R^d$ to $y\in\R^d$ is $\|x-y\|_2^2$; see \cite{peyre2019computational} for a comprehensive introduction, applications, and related discussion.

Of particular importance to our development is the Wasserstein distance between empirical measures of the same size, which can be written as $\mu_k=\frac{1}{k}\sum_{i=1}^k\delta_{x_i}$ and $\nu_k=\frac{1}{k}\sum_{j=1}^k\delta_{y_j}$ for some $\{x_i\}_{i=1}^k,\{y_j\}_{j=1}^k\subset\R^d$.  In this case, the transport problem \eqref{eq:w2} becomes a linear assignment problem with cost $C_{ij}:=\|x_i-y_j\|_2^2$:
\begin{equation}\label{eq:w2finite}
    \W_2^2(\mu_k,\nu_k)=\left\{
    \begin{array}{rl}
    \min_{T\in\R^{k\times k}} & \langle T,C\rangle\\
    \st & T\1=\nicefrac{\1}{k}\\
    & T^\top\1=\nicefrac{\1}{k}\\
    & T\geq0,
    \end{array}
    \right.
\end{equation}
where $\1$ denotes the vector of all ones.  The constraints of \eqref{eq:w2finite} form a scaled version of the Birkhoff polytope (set of doubly-stochastic matrices), whose vertices define bijections between the $x_i$'s and the $y_j$'s.

There is a probabilistic link between \eqref{eq:w2} and \eqref{eq:w2finite}.  For general $\mu,\nu\in\Prob(\R^d)$, we can define an \emph{empirical (plug-in) estimator} of $\W_2^2(\mu,\nu)$ by drawing $x_1,\ldots,x_k\sim\mu$ and $y_1,\ldots,y_k\sim\nu$ and approximating $\W_2^2(\mu,\nu)\approx \W_2^2(\mu_k,\nu_k)$ as in \eqref{eq:w2finite}.  As derived in \cite[Theorem 2]{chizat2020faster}, under straightforward assumptions this approximation converges with rate $k^{-\nicefrac2d}$ for large $k$ when $d>4$.

\section{$k$-variance}\label{sec:kvar}

We can introduce optimal transport into the variance formula \eqref{eq:variance_pairwise} using the $k=1$ case of \eqref{eq:w2finite}, by writing  $\|x-y\|_2^2=\W_2^2(\delta_x,\delta_y).$  That is, an equivalent formula to \eqref{eq:variance_pairwise} is the following: $\Var(X)=\E_{X,Y\sim\mu}[\W_2^2(\delta_X,\delta_Y)].$  This observation immediately suggests a generalization of variance using optimal transport:
\begin{definition}[$k$-variance]\label{def:kvar}
Given a probability measure $\mu\in\Prob(\R^d)$ and a parameter $k\in\N$, define \emph{$k$-variance} as
\begin{equation}\label{eq:vark}
    \Var_k(\mu):=\frac12\cdot \rho(k,d)\cdot\E_{\substack{X_1,\ldots,X_k\sim\mu \\ Y_1,\ldots,Y_k\sim\mu}}\left[\W_2^2\left(\frac{1}{k}\sum_{i=1}^k \delta_{X_i},\frac{1}{k}\sum_{i=1}^k \delta_{Y_i}\right)\right],
\end{equation}
where $\rho(k,d)$ is the \emph{ambient scaling rate} chosen to account for the rate at which the expectation approaches zero:
\begin{equation}
    \rho(k,d) := \left\{
    \begin{array}{ll}
    k & \textrm{ if }d=1\\
    \nicefrac{k}{\log k} & \textrm{ if }d=2\\
    k^{\nicefrac2d} & \textrm{ if }d>2.
    \end{array}
    \right.
\end{equation}
We define
\begin{equation}\label{eq:var_inf}
    \Var_\infty(\mu):=\lim_{k\to\infty} \Var_k(\mu),
\end{equation}
when such a limit exists. For $X\sim\mu$, we will identify $\Var_k(X):=\Var_k(\mu)$.
\end{definition}
See \cref{sec:highdimensional} for formulas motivating our choice of $\rho(k,d)$.

Several simple properties of $\Var_k(\cdot)$ in analogy to variance follow from definitions and simple properties of $\W_2$:
\begin{proposition}[Basic properties of $\Var_k(\cdot)$]
We have the following properties for $\Var_k$:
\begin{enumerate}[label=(\alph*)]
    \item $\Var_1(\mu)=\Var(\mu)$, for $\mu\in\Prob(\R^d)$\label{var_generalized}
    \item $\Var_k(\delta_a)=0$, for $a\in\R^d$\label{var_delta}
    \item $\Var_k(X+a)=\Var_k(X)$, for $X\sim\mu\in\Prob(\R^d),a\in\R^d$\label{var_shift}
    \item $\Var_k(c\cdot X)=c^2\cdot\Var_k(X)$, for $X\sim\mu\in\Prob(\R^d),c\in\R$\label{var_scale}
    \item $\Var_k(X+\tilde X)\geq\Var_k(X)+\Var_k(\tilde X)$, for independent $X\sim\mu\in\Prob(\R^d), \tilde X\sim\nu\in\Prob(\R^d)$\label{var_sum}
\end{enumerate}
\end{proposition}
\begin{proof}
Property \ref{var_generalized} is argued above.  Properties \ref{var_delta}, \ref{var_shift}, and \ref{var_scale} follow from simple properties of the cost matrix in \eqref{eq:w2finite} after substituting \eqref{eq:vark}. To prove \ref{var_sum}, we resort to the form \eqref{eq:w2finite}.  In this case, we can write
$$
\Var_k(X+\tilde X):=\frac12\cdot \rho(k,d)\cdot\E_{
\substack{X_1,\ldots,X_k\sim\mu; \tilde X_1,\ldots,\tilde X_k\sim\nu \\
Y_1,\ldots,Y_k\sim\mu; \tilde Y_1,\ldots,\tilde Y_k\sim\mu}}
\left[\W_2^2\left(\frac{1}{k}\sum_{i=1}^k \delta_{X_i+\tilde X_i},\frac{1}{k}\sum_{i=1}^k \delta_{Y_i+\tilde Y_i}\right)\right].
$$
The cost matrix of the linear program \eqref{eq:w2finite} in this expectation has entries
$$C_{ij}=\|X_i+\tilde X_i-Y_i-\tilde Y_i\|_2^2 = \|X_i-Y_i\|_2^2+\|\tilde X_i - \tilde Y_i\|_2^2+2(X_i-Y_i)\cdot(\tilde X_i-\tilde Y_i).$$
Splitting the minimization in \eqref{eq:w2finite} into three minimizations corresponding to the terms in our expression for $C_{ij}$ above shows:
\begin{align*}
\Var_k(X+\tilde X)\geq \Var_k(X) + \Var_k(\tilde X) + 2\rho(k,d)\E\Bigg[
    \min_{T\in\mathcal B_k} \sum_{ij} [(X_i-Y_i)\cdot(\tilde X_i-\tilde Y_i)]T_{ij}
    \Bigg],
    \end{align*}
    where $\mathcal B_k$ indicates the constraint set in \eqref{eq:w2finite}. 
By Jensen's inequality,
    \begin{align*}
\Var_k(X+\tilde X)&\geq \Var_k(X) + \Var_k(\tilde X) + 2\rho(k,d)\Bigg[
    \min_{T\in\mathcal B_k} \sum_{ij} \E[(X_i-Y_i)\cdot(\tilde X_i-\tilde Y_i)]T_{ij}
    \Bigg]\\
    &=\Var_k(X) + \Var_k(\tilde X)\textrm{ by independence, yielding \ref{var_sum}.}
\end{align*}
\end{proof}

In the following sections, we seek to provide intuition for $\Var_k(\cdot)$ in various settings.  We organize our discussion around \emph{dimensionality}, starting with one-dimensional measures, proceeding to measures with low-dimensional structures, and then considering the high-dimensional case.  We conclude our theoretical discussion with another structured class of measures, those containing clusters of high probability.

\section{One-dimensional $k$-variance}\label{sec:1d}

The $k$-variance $\Var_k$ admits a particularly clean formulation for probability measures over the real numbers $\R$.  Here, we derive this alternative interpretation of $\Var_k$, show how it can be used to derive bounds and estimates describing the behavior of one-dimensional $k$-variance, and give a limiting formula as $k\to\infty.$

\subsection{Alternative formula}

In one dimension, the 2-Wasserstein distance $\W_2$ between empirical measures consisting of the same number of points is given by the $L^2$ distance between the vectors of data points \cite{villani2003topics}.  That is,
\begin{equation}\label{eq:w2in1d}
    \W_2\left(\frac{1}{k}\sum_{i=1}^k \delta_{x_i},\frac{1}{k}\sum_{i=1}^k \delta_{y_i}\right)=\sqrt{\sum_{i=1}^k (x_i-y_i)^2},
\end{equation}
when $x_1\leq x_2\leq\cdots\leq x_k$ and $y_1\leq y_2\leq\cdots\leq y_k.$

To incorporate this formula into \eqref{eq:vark}, take $X_{(i)}$ to be the $i$-th \emph{order statistic} of $X_1,\ldots,X_k\sim\mu\in\Prob(\R)$, a random variable obtained by sorting $\{X_i\}_{i=1}^k$ and taking the $i$-th element of the sorted list; similarly define order statics $Y_{(i)}$ for the samples $\{Y_i\}_{i=1}^k$.  Then, for $d=1$ we can write
\begin{equation}\label{eq:kvarfromorderstatistics}
    \Var_k(\mu)
    =\frac12\cdot \E_{\substack{X_1,\ldots,X_k\sim\mu \\ Y_1,\ldots,Y_k\sim\mu}}\left[\sum_{i=1}^k (X_{(i)}-Y_{(i)})^2\right]
    =\sum_{i=1}^k \Var(X_{(i)}),
\end{equation}
by linearity of expectation and by applying \eqref{eq:w2in1d} and \eqref{eq:variance_pairwise}.  Hence, in one dimension, the $k$-variance is exactly the sum of the variances of the order statistics.

\begin{example}[Uniform distribution]\label{ex:uniform}
Suppose $\mu$ is the uniform distribution on the unit interval.  Then, $X_{(i)}\sim\mathrm{Beta}(i,k+1-i).$
Hence, 
\begin{align}
    \E(X_{(i)}) &= \frac{i}{k+1}\label{eq:unifexpectation}\\
    \Var(X_{(i)})&=\frac{i (k+1-i)}{(k+1)^2 (k+2)}=\frac{p_i(1-p_i)}{k+2}\label{eq:unifvar}\\
    \E( (X_{(i)}-p_i)^4 ) &=\frac{3 i (k-i+1)[2(k+1)^2 + i(k-i+1)(k+5)]}{(k+1)^4(k+2)(k+3)(k+4)}\nonumber\\
&=\frac{3p_i(1-p_i)[2+p_i(1-p_i)(k+5)]}{(k+2)(k+3)(k+4)},\label{eq:fourthpowerunif}
\end{align}
where $p_i=\nicefrac{i}{k+1}$; we include some of the expressions above to assist in our proof of \cref{prop:1dtaylor}. Substituting \eqref{eq:unifvar} into our expression for one-dimensional $k$-variance, 
\begin{equation}\label{eq:unifvark}
    \Var_k(\mathrm{Unif}([0,1]))
=\sum_{i=1}^k \Var(X_{(i)})=\frac{1}{(k+1)^2(k+2)}\sum_{i=1}^k i(k+1-i)=\frac{k}{6(k+1)}.
\end{equation}
This sequence is increasing, and taking a limit as $k\to\infty$ shows $\Var_\infty(\mathrm{Unif}([0,1]))=\nicefrac{1}{6}.$
\end{example}

\begin{example}[Exponential distribution]\label{ex:exponential} 
Suppose $\mu$ is an exponential distribution with parameter $\lambda$.  Then, we can sample from the order statistics of $\mu$ by drawing iid exponential variables $Z_j$ with rate 1 and computing the following \cite{renyi1953theory}:
$$X_{(i)}=\frac{1}{\lambda}\sum_{j=1}^i \frac{Z_j}{k-j+1}.$$
Substituting the variance of an exponential random variable,
$$\Var(X_{(i)})=\sum_{j=1}^i \left(\frac{1}{\lambda(k-j+1)}\right)^2.$$
This gives the following expression for $k$-variance:
$$
\Var_k(\mathrm{Exp}(\lambda))=\frac{1}{\lambda^2}\sum_{i=1}^k \sum_{j=1}^i \frac{1}{(k-j+1)^2} 
=\frac{H_k}{\lambda^2}\approx \log(k) +\gamma,
$$
where $H_k$ is the $k$-th harmonic number and $\gamma$ is the Euler's constant.  Taking $k\to\infty$ shows $\Var_\infty(\mathrm{Exp}(\lambda))=\infty.$
\end{example}

\subsection{Properties of $k$-variance in 1D}

We can immediately derive alternative expressions/bounds for $\Var_k(\cdot)$ in one dimension by applying properties of order statistics:
\begin{proposition}[Bounding $\Var_k(\cdot)$ in 1D]\label{prop:1dbound}
When $d=1$, we can write
\begin{equation}\label{eq:vark1dalternative}
    \Var_k(\mu)=k\sigma^2-\sum_{i=1}^k \left(\overline X_{(i)} - \overline X\right)^2\leq k\sigma^2.
\end{equation}
Moreover, we can bound
\begin{equation}\label{eq:varbound}
    \Var_k(\mu)\geq k\sigma^2-2\sum_{i<j} \sigma_{(i)}\sigma_{(j)} \cdot\frac{i(k+1-j)}{j(k+1-i)},
\end{equation}
with equality for uniform distributions. In these expressions, $X\sim\mu$, $\sigma^2=\Var(X)$, and $\sigma_{(i)}^2=\Var(X_{(i)})$ for $X_1,\ldots,X_k\sim\mu$.
\end{proposition}
\begin{proof}
We can obtain \eqref{eq:vark1dalternative} by rearranging a sum:
\begin{align}
k\sigma^2=\sum_{i=1}^k\E[(X_i-\overline X)^2]=\sum_{i=1}^k\E[(X_{(i)}-\overline X)^2]&=\sum_{i=1}^k\E[(X_{(i)}-\overline X_{(i)}+\overline X_{(i)}-\overline X)^2]\nonumber\\
&=\Var_k(\mu) + \sum_{i=1}^k (\overline X_{(i)}-\overline X)^2.\label{eq:varianceboundfinalstep}
\end{align}
Removing the final term provides inequality \eqref{eq:vark1dalternative}.

To derive \eqref{eq:varbound}, we rely on a bound on the correlation of order statistics stated in \cite[p.\ 74]{david2003order} and references therein.  In particular, for $i<j$ they show:
\begin{equation}\label{eq:corrbound}
\mathrm{Corr}(X_{(i)},X_{(j)})\leq\frac{i(k+1-j)}{j(k+1-i)},
\end{equation}
where $\mathrm{Corr}(\cdot,\cdot)$ denotes the correlation of random variables, with equality when the parent distribution is uniform. 
We know $\sum_i X_{(i)}=\sum_i X_i$ given the $X_i$'s are iid variables with variance $\sigma^2$; computing the variance of both sides shows
$$k\sigma^2=\Var_k(\mu)+2\sum_{i<j}\Cov(X_{(i)},X_{(j)}).$$
Substituting \eqref{eq:corrbound}, by definition of correlation we have
$$\Var_k(\mu)=k\sigma^2-2\sum_{i<j}\Cov(X_{(i)},X_{(j)})\geq k\sigma^2 - 2\sum_{i<j} \sigma_{(i)}\sigma_{(j)}\cdot \frac{i(k+1-j)}{j(k+1-i)},$$
as needed.
\end{proof}

\begin{remark}[Approximating $\Var_k$]
The expression \eqref{eq:vark1dalternative} suggests the following means of approximating $\Var_k(\mu)$ for large $k$:
\begin{equation}
    \Var_k(\mu)\approx k\sigma^2 - \sum_{i=1}^k (F^{-1}(p_i)-\overline X)^2,
\end{equation}
where $p_i=\nicefrac{i}{k+1}$ and $F^{-1}$ is the quantile function associated to $\mu$. Intuitively, this expression indicates that our index of total local variability is approximately a global variability index minus an index of between-local-group variability.
\end{remark}

Another standard approach to working with order statistics involves Taylor series expansions about quantiles of the sampled probability measure.  Following this strategy yields a useful approximation to $\Var_k(\cdot)$ as well as a limiting formula under certain assumptions about the distribution function:
\begin{proposition}\label{prop:1dtaylor}
Using the notation of \cref{prop:1dbound}, suppose that $\sigma^2$ is finite and that $\mu$ has a differentiable distribution function $f(x)$ with CDF $F(x)$.  Moreover, suppose (i) $f(x)>0$ and (ii) $\nicefrac{f'(x)}{[f(x)]^3}$ is bounded on $F^{-1}((0,1))$. Then, 
as $k\to\infty$ we have
\begin{equation}\label{eq:varapprox}
    \Var_k(\mu)  \approx \frac{1}{k+2} \sum _{i=1}^k  \frac{p_i (1-p_i)}{[f(F^{-1}(p_i))]^2},
\end{equation}
where $p_i=\nicefrac{i}{k+1}$. 
As $k\to\infty$, under the assumptions above we have
\begin{equation}\label{eq:integrallimit}
    \Var_k(\mu) \to \int_{0}^1\frac{u(1-u)}{[f(F^{-1}(u))]^2}\,du = \int_{F^{-1}((0,1))}\frac{F(x)(1-F(x))}{f(x)}\,dx.
\end{equation}
The rate of convergence of $\Var_k(\mu)$ to the limiting integral $\Var_\infty(\mu)$ is of $O(\nicefrac{1}{\sqrt{k}})$.
\end{proposition}
\begin{proof}
Note that $X_{(i)}\stackrel{d}{=}F^{-1}(U_{(i)})$ where $U_{(i)}$ is the $i$-th order statistic from the standard uniform parent. We begin with a Taylor expansion for $F^{-1}(U_{(i)})$ given in \cite{arnold1989approximations}. With $p_i=\nicefrac{i}{k+1}$,
\begin{equation}\label{eq:taylor}
F^{-1}(U_{(i)}) = F^{-1}(p_i) + (U_{(i)}-p_i) (F^{-1}(p_i))' + \frac{1}{2}(U_{(i)}-p_i)^2 (F^{-1}(V_i))'',
\end{equation}
for some random variable $V_i\in (p_i, U_{(i)})$.  Differentiating inverse functions shows
\begin{equation}
(F^{-1}(u))' = \frac{1}{f(F^{-1}(u))}\qquad\textrm{ and }\qquad
(F^{-1}(u))'' =  -\frac{f'(F^{-1}(u))}{[f(F^{-1}(u))]^3}.
\end{equation}
Substituting into \eqref{eq:taylor} and taking variance of both sides shows
\begin{equation}
\begin{array}{rl}
\sigma_{(i)}^2 = &\displaystyle \Var(U_{(i)}-p_i)[f(F^{-1}(p_i))]^{-2} +\frac14\Var((U_{(i)}-p_i)^2\cdot (F^{-1}(V_i))'')\\
&\displaystyle +\frac12f(F^{-1}(p_i))^{-1}\Cov(U_{(i)}-p_i,(U_{(i)}-p_i)^2 (F^{-1}(V_i))''),
\end{array}\label{eq:sigmaitaylor}
\end{equation}
where $\Var_k(\mu) = \sum_{i=1}^k \sigma_{(i)}^2$. 

Applying the identity $\Var[Y]\leq\E[Y^2]$, the variance factor in the second term of \eqref{eq:sigmaitaylor} is bounded above by $E((U_{(i)}-p_i)^4 [(F^{-1}(V_i))'']^2)$, which in turn is bounded by $M^2\cdot\E((U_{(i)}-p_i)^4)$ where $M$ is an upper bound for $(F^{-1}(u))''$ for $u \in (0,1)$. From \eqref{eq:fourthpowerunif}, we obtain
\begin{eqnarray*}\label{eq9}
\sum_{i=1}^k\E((U_{(i)}-p_i)^4) &=& \sum_{i=1}^k\frac{3p_i(1-p_i)[2+p_i(1-p_i)(k+5)]}{(k+2)(k+3)(k+4)}\nonumber\\
&=& \frac{6}{(k+3)(k+4)}\sum_{i=1}^k \frac{p_i(1-p_i)}{(k+2)}+ \frac{3(k+5)}{(k+3)(k+4)}\sum_{i=1}^k \frac{p_i^2(1-p_i)^2}{(k+2)}\nonumber\\
&\approx&\frac{6}{k^2}\int_{0}^{1}u(1-u)\,du+ \frac{3}{k}\int_{0}^{1}u^2(1-u)^2\,du = \frac{1}{k^2} + \frac{1}{10k}.
\end{eqnarray*}
Here, the ratios of the first and second terms on the right and left side of the approximation $\approx$ each approach $1$ for large $k$. Thus, we conclude that when summed over $i$, the second term in \eqref{eq:sigmaitaylor} contributes an amount of size $O(\nicefrac1k)$.

The covariance term in \eqref{eq:sigmaitaylor} can be bounded as follows
\begin{eqnarray}
\Cov(\cdots)
 &\leq& [\Var(U_{(i)})\Var((U_{(i)}-p_i)^2F^{-1}(V_i))'')]^{1/2}\nonumber\\
&\leq& [\Var(U_{(i)})]^{1/2}[\E((U_{(i)}-p_i)^4 {M^2}]^{1/2}\nonumber\\
&=&M\cdot\left[\frac{p_i(1-p_i)}{k+2}\cdot\left[\frac{6p_i(1-p_i)}{(k+2)(k+3)(k+4)}+\frac{3(k+5)p_i^2(1-p_i)^2}{(k+2)(k+3)(k+4)}\right]\right]^{1/2}\nonumber\\
&=&M\cdot\frac{p_i(1-p_i)}{k+2}\cdot\left[\frac{6}{(k+3)(k+4)}+\frac{3(k+5)p_i(1-p_i)}{(k+3)(k+4)}\right]^{1/2}\nonumber\\
&<& M\cdot\frac{p_i(1-p_i)}{k+2}\cdot\frac{C}{\sqrt{k+3}}\nonumber
\end{eqnarray} 
for some constant $C$ ($C = 3$ suffices).  Summing over $i$,
\begin{align*}
&\sum_{i=1}^k\frac12f(F^{-1}(p_i))^{-1}\Cov(U_{(i)}-p_i,(U_{(i)}-p_i)^2 (F^{-1}(V_i))'')\\
&\leq \frac{M C}{2\sqrt{k+3}}\frac{1}{k+2}\sum_{i=1}^k \frac{p_i(1-p_i)}{f(F^{-1}(p_i))}
\approx \frac{M C}{2\sqrt{k}}\int_{0}^{1}\frac{u(1-u)}{f(F^{-1}(u))}\,du
= \frac{MC}{2\sqrt{k}}\int_{F^{-1}((0,1))}\hspace{-.3in}F(x)(1-F(x))\,dx,
\end{align*}
where the equality follows upon using the transformation $u=F(x)$.
If the support of $F$ is bounded, the integral above is always finite.  Even when the support is infinite, the integral is finite whenever the variance or the second moment of $F$ is finite, by the comparison test:  Finiteness of the variance implies that as $x \to \infty$, $x^2(1-F(x)) \to 0$, and as $x \to - \infty$, $x^2F(x) \to 0$. Consequently, the covariance sum  is of $O(\nicefrac{1}{\sqrt{k}})$.

Summing the first term in \eqref{eq:sigmaitaylor} over $i$, using \eqref{eq:unifexpectation} we find
 \begin{align*}
 \sum_{i=1}^k [f(F^{-1}(p_i))]^{-2}\Var(U_{(i)}-p_i) 
 &=  \frac{1}{k+2}\sum_{i=1}^k \frac{p_i(1-p_i)}{[f(F^{-1}(p_i))]^2},\textrm{ validating \eqref{eq:varapprox}}\\
 &\approx \int_{0}^{1}\frac{u(1-u)}{[f(F^{-1}(u))]^2}\,du.
 \end{align*}
The transformation $u = F(x)$ shows that
$$\int_{0}^{1}\frac{u(1-u)}{[f(F^{-1}(u))]^2}\,du = \int_{F^{-1}((0,1))}\frac{F(x)(1-F(x))}{f(x)}\,dx,$$
as desired.
\end{proof}

\begin{remark}[Relationship to \cite{bobkov2019one}]
In \cite{bobkov2019one}, Bobkov and Ledoux provide a comprehensive discussion of one-dimensional optimal transport from samples in an attempt to understand convergence of empirical approximations to a measure in the Wasserstein metric.  Their analysis focuses on the ``one-sided'' convergence of an empirical approximation to a true measure, while $k$-variance is based on the Wasserstein distance between two different empirical approximations.

That said, along the way their discussion does make some similar observations to our discussion above.  For instance, their Theorem 4.3 shows the same link to order statistics as our \eqref{eq:kvarfromorderstatistics}.  The ``$J_2$ functional'' defined in their (5.3) is the right-hand side of \eqref{eq:integrallimit}; in our notation, their Theorem 5.1 (and, in particular, their Corollary B.6) implies a bound
\begin{equation}
    \Var_k(\mu)\leq\frac{k}{k+1}J_2(\mu).
\end{equation}
This establishes half of our equality in \eqref{eq:integrallimit}.  Their results show $\lim\sup_{k\to\infty} \Var_k(\mu)\leq J_2(\mu)$, while we are able to show under stronger assumptions that $\lim_{k\to\infty} \Var_k(\mu)=J_2(\mu).$
\end{remark}

\begin{example}[Uniform distribution, continued] 
Continuing \cref{ex:uniform}, we can apply \eqref{eq:integrallimit} to compute
\begin{equation}
    \Var_\infty(\Unif([0,1]))=\int_0^1 \frac{x(1-x)}{1}\,dx=\frac{1}{6}.
\end{equation}
As expected, this expression agrees with \eqref{eq:unifvark} as $k\to\infty.$
\end{example}

\begin{example}[Weibull distribution with shape parameter $\alpha$]
For this distribution, $F(x) = 1- \exp\{-x^{\alpha}\}$ and $f(x) = \alpha x ^{\alpha -1}\exp\{-x^{\alpha}\}$ for $x > 0$, with shape parameter $\alpha > 0$. As $x \to 0^+$,
$$\frac{F(x)(1-F(x))}{f(x)} = \frac{1- \exp\{-x^{\alpha}\}}{\alpha x^{\alpha-1}}\approx\frac{x^{\alpha}}{\alpha x^{\alpha-1}}= \frac{x}{\alpha},$$
and consequently the integral \eqref{eq:integrallimit} is always convergent at the lower limit of integration.  
As $x \to \infty$,
$$\frac{F(x)(1-F(x))}{f(x)} = \frac{1- \exp\{-x^{\alpha}\}}{\alpha x^{\alpha-1}}\approx\frac{1}{\alpha x^{\alpha-1}},$$
and hence the integral \eqref{eq:integrallimit} is convergent at the upper limit if and only if $\alpha >2$.  

Now, for $\alpha >2$, \eqref{eq:integrallimit} implies
\begin{align*}
\Var_\infty(\mathrm{Weib}(\alpha)) &= \int_{0}^{\infty}\frac{1- \exp\{-x^{\alpha}\}}{\alpha x^{\alpha-1}}\,dx
= \frac{1}{\alpha^2}\int_{0}^{\infty}(1-e^{-y})y^{2/\alpha-2}\,dy\\
&= \frac{1}{\alpha^2(2/\alpha-1)}\int_{0}^{\infty}(1-e^{-y})\,d(y^{2/\alpha-1}).
\end{align*}
Upon integration by parts we see that
$$\int_{0}^{\infty}(1-e^{-y})\,d(y^{2/\alpha -1}) = \left[(1-e^{-y})y^{2/\alpha -1}\right]_{0}^{\infty} -\int_{0}^{\infty}e^{-y}y^{2/\alpha -1}\,dy.$$
For $\alpha >2$, the first term yields 0 at both upper and lower limits, and the second term equals $\Gamma(\nicefrac{2}{\alpha})$.  Thus, 
$$\Var_\infty(\mathrm{Weib}(\alpha)) = \left\{
\begin{array}{ll}
\frac{\Gamma(\nicefrac{2}{\alpha})}{\alpha(\alpha-2)} & \textrm{ when }\alpha>2\\
\infty & \textrm{ when }\alpha\leq2.
\end{array}
\right.$$
\end{example}

\begin{example}[Tukey's symmetric $\lambda$ distribution]
This distribution is defined by its quantile function $F^{-1}(u)$, given by
\begin{equation}
F^{-1}(u) = \left\{\begin{array}{ll}
\frac{1}{\lambda}(u^{\lambda} -(1-u)^{\lambda}) & \textrm{ when } \lambda \neq 0\\
\log(\nicefrac{u}{(1-u)}) & \textrm{ when } \lambda = 0
\end{array}\right.
\end{equation}
for $u\in[0,1]$ and $\lambda\in\R$.  When $\lambda=0$, we obtain the standard logistic distribution. 

When $\lambda \neq 0$, the quantile density function $(F^{-1}(u))'$ is given by $u^{\lambda-1} + (1-u)^{\lambda-1}$,  and $$(F^{-1}(u))'' = (\lambda -1)[u^{\lambda-2} - (1-u)^{\lambda -2}] $$ is bounded if and only if $\lambda \geq 2$. Hence, we satisfy the sufficient conditions needed for \cref{prop:1dtaylor}. Thus for $\lambda \geq2$,
\begin{align} 
\Var_\infty(\mathrm{Tukey}(\lambda))
&= \int_{0}^{1} u(1-u)(F^{-1}(u))'\,du
= \int_{0}^{1} u(1-u)[u^{\lambda -1} + (1-u)^{\lambda -1}]\,du\nonumber\\
&= 2\cdot\mathrm{Beta}(\lambda +1, 2)
= \frac{2}{(\lambda+1)(\lambda+2)}.\label{eq:tukey} 
\end{align}
The integral on the right is finite whenever $\lambda>-1$, and the expression holds for $\lambda=0$.  For $\lambda\in(-1,2),$ we can only say that $\lim\sup_{k\to\infty}\Var_k(\mu)$ is bounded above by the right-hand side.
\end{example}

\section{Low-dimensional measures}\label{sec:lowdimensional}

Having worked out the case of one-dimensional measures, we now consider measures that have low-dimensional structure but are embedded in a higher-dimensional space.  
Specifically, define
\begin{definition}[$\epsilon$-fattening and $\epsilon$-covering number, \cite{talagrand1995concentration,weed2019sharp}]
For any compact $X\subset\R^d$ and $S\subseteq X$, the $\epsilon$-fattening of $S$ is $S_\epsilon := \{y: D(y,S) \leq \epsilon\}$, where $D$ denotes the Euclidean distance.  The $\epsilon$-covering number $\mathcal N_\epsilon(S)$ of $S$ is the minimum $m$ such that there exist $m$ points $x_1,\ldots,x_m\in\R^d$ with $S\subseteq\bigcup_i B_\epsilon(X_i).$
\end{definition}
We borrow a recent bound on empirical transport, specialized to $\W_2$:
\begin{proposition}[\cite{weed2019sharp}, Proposition 15]
Suppose $\supp(\mu) \subseteq S_\epsilon$ for some $\epsilon > 0$, where $S$ satisfies $\mathcal{N}_{\epsilon'}(S) \leq (3\epsilon')^{-d'}$ for all $\epsilon' \leq \nicefrac{1}{27}$ and some $d'>4$. Then, for all $k \leq (3\epsilon)^{-d'}$, we have $\E[\W_2^2(\mu, \hat{\mu}_k)]\leq C_1\cdot k^{-2/d'}$, where $C_1 = 27^2 (2 + 1/(3^{d'/2 - 2} - 1))$ and $\mu_k$ denotes the $k$-point empirical measure.
\end{proposition}
Translating this to our setting using the triangle inequality, we get 
\begin{proposition}[$\Var_k(\cdot)$ for low-dimensional distributions]\label{prop:varklowdim}
Suppose $\supp(\mu) \subseteq S_\epsilon$ for some $\epsilon > 0$, where $S$ satisfies $\mathcal{N}_{\epsilon'}(S) \leq (3\epsilon')^{-d'}$ 
for all $\epsilon' \leq \nicefrac{1}{27}$ and some $d' > 4$. Then, for all $k \leq (3\epsilon)^{-d'}$, we have $\Var_k(\mu)\leq C_1 \cdot k^{2/d-2/d'}.$
\end{proposition}
Unsurprisingly, the proposition above shows that if we measure the $d$-dimensional $k$-variance of an intrinsically $d'$-dimensional measure, at least when $4<d'<d$ we have $\Var_k(\mu)\to0$ as $k\to\infty$.  As a special case, we see that empirical measures have $k$-variance tending to zero for higher-dimensional measures.  Interestingly, this is not the case in low dimensions, as we can see in the following example:
\begin{example}[Two-point empirical measures]
Take $\mu=\nicefrac{(\delta_{-0.5e_1}+\delta_{0.5e_1})}{2}$, constructed from standard basis vector $e_1=(1,0,\ldots,0)\in\R^d$.  In this case, $\W_2$ between two $k$-samples from $\mu$ counts the imbalance in the number of $-0.5$ vs.\ $0.5$ samples between the two draws.  Hence,  $\Var_k(\mu)$ is the expected absolute difference $|A-B|$ of two binomial variables $A,B\sim B(k,\nicefrac12)$, scaled by $\nicefrac{\rho(k,d)}{2k}$.  From \cite[eq.\ (2.9)]{ramasubban1958mean}, for binomially-distributed variables $X_1,X_2\sim B(k,p)$ we have
$$\E(|X_1-X_2|)
=2kp(1-p)\cdot{}_2F_1\left(1-k,\frac12;2;4p(1-p)\right),$$
where ${}_2F_1$ is Gauss' hypergeometric function.  Substituting $p=\nicefrac12$ shows 
$$
\E(|X_1-X_2|)
\!=\!\frac{k\Gamma(2)}{2\Gamma(\nicefrac12)\Gamma(\nicefrac32)}\int_0^1\!\!t^{-\nicefrac12}(1-t)^{k-1}\,dt
\!=\!\frac{k\Gamma(k+\nicefrac12)}{2\Gamma(\nicefrac32)\Gamma(k+1)}
\!=\!{{2k}\choose{k}}\cdot\frac{k}{2^{2k}}.
$$
Hence,
$$
\Var_k(\mu)=\frac{\rho(k,d)}{2^{2k+1}} {{2k}\choose{k}}
\approx \frac{\rho(k,d)}{2\sqrt{\pi k}}
=\frac{1}{2\sqrt{\pi}}\cdot\left\{
\begin{array}{ll}
\sqrt{k} & \textrm{ if }d=1\\
(\log k)^{-1}\cdot\sqrt{k} & \textrm{ if }d=2\\
k^{2/d-1/2} & \textrm{ if }d\geq3.
\end{array}
\right.
$$
by Stirling's approximation.  So, $\Var_k(\mu)$ diverges for $d\leq3$, converges to $\nicefrac{1}{2\sqrt{\pi}}$ for $d=4$, and converges to $0$ for $d\geq5$.
\end{example}

\section{Higher-dimensional measures}\label{sec:highdimensional}

A surprising result of our experiments detailed in \cref{sec:experiments} is that one-dimensional $k$-variance seems to have totally different behavior than $k$-variance for measures on $\R^d$ for large $d$.  While we cannot provide as a complete a story as \cref{sec:1d} for the one-dimensional case, some results in the theory of random Euclidean matching are directly relevant to our construction and can provide some insight into the behavior of $\Var_k(\cdot)$.  

\begin{example}[Unit cube]
Suppose $\mu=\Unif([0,1]^d)$.  Then, for large $k$ we have the following formula \cite[eq.\ (1.1)]{goldman2020convergence}:
\begin{equation}\label{eq:uniformcube}
\E_{\substack{X_1,\ldots,X_k\sim\mu \\ Y_1,\ldots,Y_k\sim\mu}}\left[\W_2^2\left(\frac{1}{k}\sum_{i=1}^k \delta_{X_i},\frac{1}{k}\sum_{i=1}^k \delta_{Y_i}\right)\right]\approx \left\{\begin{array}{rl}
k^{-1} & \textrm{ if }d=1\\
\nicefrac{(\log k)}{k} & \textrm{ if }d=2\\
k^{-\nicefrac2d} & \textrm{ if }d\geq3.
\end{array}\right.
\end{equation}
These formulas motivate our choice of scaling factors $\rho(k,d)$ in \cref{def:kvar}.  \cite[Theorem 2]{chizat2020faster} observes similar rates for $d>4$ for general measures with support in the unit ball, but their upper bound decays more slowly in $k$ than \eqref{eq:uniformcube} for $d\leq4$.
\end{example}

\begin{example}[Unit square]
\cite{benedetto2020euclidean} predicts a similar $\nicefrac{(\log k)}{k}$ rate for measures with positive density on the unit square $[0,1]^2$.  Specifically for $\mu=\Unif([0,1]^2)$, we can obtain the following limit \cite[Theorem 1.1]{ambrosio2019finer}:
$$\lim_{k\to\infty} \frac{k}{\log k}\E_{\substack{X_1,\ldots,X_k\sim\mu \\ Y_1,\ldots,Y_k\sim\mu}}\left[\W_2^2\left(\frac{1}{k}\sum_{i=1}^k \delta_{X_i},\frac{1}{k}\sum_{i=1}^k \delta_{Y_i}\right)\right]=\frac{1}{2\pi}.$$
Hence, we have $\Var_\infty([0,1]^2)=\nicefrac1{2\pi}.$
\end{example}

\section{Clustered measures}\label{sec:clustered}

To give an idea of the value of measuring $k$-variance for $k>1$, in this section we explore the case of \emph{clustered} measures, which distinguishes the behavior of $\Var_k(\cdot)$ from that of $\Var_1(\cdot)$. 
We consider the following definitions, again from \cite{weed2019sharp} similar to our discussion in \cref{sec:lowdimensional}, which provide two ways of identifying clusterable structure in probability measures:
\begin{definition}[$(m,\sigma^2)$-Gaussian mixture]
A distribution $\mu$ is an \emph{$(m,\sigma^2)$-Gaussian mixture} if it is a mixture of $m$ Gaussian distributions in $\mathbb{R}^d$, and the trace of the covariance matrix of each mixture component is bounded above by $\sigma^2$. 
\end{definition}
\begin{definition}[Clusterable measure]
A distribution $\mu$ is \emph{$(m,\Delta)$-clusterable} if $\mathrm{supp}(\mu)$ lies in the union of $m$ balls of radius at most $\Delta$.
\end{definition}

The following proposition from \cite{weed2019sharp} directly suggests a $k$-variance bound:
\begin{proposition}[\cite{weed2019sharp}, Propositions 13 and 14]\label{prop:weedclustered}  If $\mu$ is a $(m,\sigma^2)$-Gaussian mixture and $\log \nicefrac{1}{\sigma}\geq \nicefrac{25}{8}$, then for all $k \leq m (32 \sigma^2 \log \nicefrac{1}{\sigma})^{-2}$,
 \begin{equation}\label{eq:weedbach}
 \E[\W_2^2(\mu, \hat{\mu}_k)] \leq 84 \sqrt{\nicefrac{m}{k}},
\end{equation}
where $\hat \mu_k$ is the empirical measure obtained by drawing $k$ samples. 
  The same rate holds for $(m,\Delta)$-clusterable distributions for all $k \leq m (2\Delta)^{-4}$.
\end{proposition}

Application of the triangle inequality to \cref{prop:weedclustered} immediately yields the following:
\begin{proposition}[$\Var_k$ for clustered distributions]
Suppose $d> 4$. For the $(m,\sigma^2)$-Gaussian mixture case with $k \leq m (32 \sigma^2 \log \nicefrac{1}{\sigma})^{-2}$:
\[\Var_k(\mu) \leq {\frac{168m^{1/2}}{k^{1/2 - 2/d}}}.\]
For the $(m,\Delta)$-clusterable case with $k \leq m(2\Delta)^{-4}$: 
\[
\Var_k(\mu)
\leq {\frac{168m^{1/2}}{k^{1/2 - 2/d}}}.\]
\end{proposition}

Roughly, this proposition shows that as $d$ increases and $k$ satisfies the inequality, clustered distributions have increasingly small $\Var_k(\cdot)$, though the rate of increase slows rapidly once $d$ gets beyond $\sim\!10$.

\section{Variance of empirical $k$-variance}\label{sec:variance}
We conclude our mathematical discussion by considering the problem of how to \emph{compute} $k$-variance in practice.  There exists an extremely simple \emph{empirical estimator} directly motivated by the expectation in \eqref{eq:vark}:  simply draw $2k$ samples, solve the linear program \eqref{eq:w2finite}, and use the resulting value. Note a simple implementation of this algorithm takes roughly $O(k^2d+k^3)$ time, accounting for the time taken to compute the pairwise cost matrix as well as solving the transport linear program (our implementation uses \cite{duff2001algorithms}). Here we bound the variance of this estimator, roughly showing that fewer trials need to be averaged to compute $k$ in large dimension.

In detail, we consider the empirical estimator built from $n$ trials:
\[
\widehat{\mathrm{Var}}_k(\mu) := \frac{\rho(k,d)}{2n}\sum_{j=1}^n \W_2^2(\hat{\mu}^j_k,\hat{\mu}'^j_k),
\]
where $\hat{\mu}^j_k$, $\hat{\mu}'^j_k$ are independent empirical measures formed from $k$ i.i.d.\ samples as in \eqref{eq:vark}.

The following theorem helps characterize the variance of our estimator above:
\begin{theorem}[Empirical variance]\label{thm:varVanilla}
 Suppose $\mu\in\Prob(\R^d)$ has support in a set of radius $R$. For each $j \in\{1,\ldots, n\}$, take $\hat{\mu}^j_k,\hat{\mu}'^j_k$ to be independent empirical measures each constructed from $k$ i.i.d.\ samples from $\mu$ ($X^{k,j}:=(X^j_1,\ldots,X^j_k)$ for $\hat{\mu}_k$ and $Y^{k,j}:=(Y^j_1,\ldots,Y^j_k)$ for $\hat{\mu}'_k$). Then,
 \begin{equation}
\mathbb{P}\left(
\left|
\widehat{\mathrm{Var}}_k(\mu)
-\mathrm{Var}_k(\mu)\right| \geq  \rho(k,d)R^2 \sqrt{\frac{\log (kn)}{kn}}\right)\leq \frac{2}{k^{2}n^2}.
\end{equation}
 \end{theorem}
\begin{proof}
We use McDiarmid's inequality: 
\begin{lemma}[McDiarmid's Inequality, \cite{mcdiarmid1989method}]
Let $X^m:=(X_1,\ldots,X_m)$ be an $m$-tuple of $\mathcal{X}$-valued independent random variables. Suppose $g:\mathcal{X}^m\to\mathbb{R}$ is a map that for any $i=1,\ldots,m$ and $x_1,\ldots,x_m,x_i'\in\mathcal{X}$ satisfies
\vspace{2mm}
\begin{equation}
    \big|g(x_1,\ldots,x_m)-g(x_1,\ldots,x_{i-1},x'_i,x_{i+1},\ldots,x_m)\big|\leq c_i,\label{EQ:bdd_diff}
\end{equation}
for some non-negative $\{c_i\}_{i=1}^m$. Then for any $t>0$:
\begin{subequations}
\begin{align}
    \mathbb{P}\Big(g(X_1,\ldots,X_m)-\E g(X_1,\ldots,X_m)\geq t\Big)&\leq e^{-\frac{2t^2}{\sum_{i=1}^m c_i^2}}\label{EQ:McDiarmid1}\\
    \mathbb{P}\Big(\big|g(X_1,\ldots,X_m)-\E g(X_1\ldots,X_m)\big|\geq t\Big)&\leq 2e^{-\frac{2t^2}{\sum_{i=1}^m c_i^2}}.\label{EQ:McDiarmid2}
\end{align}
\end{subequations}
\end{lemma}

Consider $\frac{1}{n} \sum_{j = 1}^n \W_2^2(\hat{\mu}^j_k,\hat{\mu}'^j_k)$ as a function of the $nk$ independent samples from which it is computed, each sample being a pair $(x_i^j, y_i^j)$. Using Kantorovich--Rubinstein duality, we have the general formula:
\[
\W_2^2(P,Q) = \sup_{(f,g) \in \Phi}  \E_P[f] + \E_Q[g]
\]
where $\Phi = \{ (f,g)\in L^1(P)\times L^1(Q) :   f(x)+g(y)\leq \|x-y\|^2\}$.  In our case, separately for each $j$, we can write
\begin{align*}
    \W_2^2(\hat{\mu}^j_k,\hat{\mu}'^j_k) = \W_2^2\left(\frac{1}{k} \sum_{\ell = 1}^k \delta_{x^j_\ell}, \frac{1}{k} \sum_{\ell = 1}^k \delta_{y^j_\ell}\right)
    = \sup_{(f,g) \in \Phi} \frac{1}{k} \sum_{\ell = 1}^k (f(x^j_\ell)+ g(y^j_\ell)).
\end{align*}
Recall that the $(x^j_\ell, y^j_\ell)$ are independent across $\ell$ and $j$. Consider replacing one of the elements $(x^j_i, y^j_i)$ with some $(x'^j_i, y'^j_i)$, forming $\bar{\mu}^j_k$ and $\bar{\mu}'^j_k$. Since the $(x^j_\ell, y^j_\ell)$ are identically distributed, by symmetry we can set $i = 1$. We thus bound
\begin{align*}
    \W_2^2(\hat{\mu}^j_k,\hat{\mu}'^j_k) - \W_2^2(\bar{\mu}^j_k,\bar{\mu}'^j_k)
    &= \sup_{(f,g) \in \Phi} \frac{1}{k}\left((f(x^j_1)+ g(y^j_1)) + \sum_{\ell = 2}^k (f(x^j_\ell)+ g(y^j_\ell))\right)\\
    &\qquad- \sup_{(f,g) \in \Phi} \frac{1}{k}\left((f(x'^j_1)+ g(y'^j_1)) + \sum_{\ell = 2}^k (f(x^j_\ell)+ g(y^j_\ell))\right)
    \leq  \frac{2R^2}{k},
\end{align*}
where we have assumed the space is bounded with radius $R$ and used the definition of $\Phi$ and \cite[Remark 1.13]{villani2003topics}. 

Hence by symmetry and scaling by $\frac{\rho(k,d)}{2n}$ as in the expression in the theorem we have
\[
\left|\frac{\rho(k,d)}{2n}\W_2^2(\hat{\mu}^j_k,\hat{\mu}'^j_k) - \frac{\rho(k,d)}{2n}\W_2^2(\bar{\mu}^j_k,\bar{\mu}'^j_k)\right| \leq \frac{\rho(k,d)R^2}{kn},
\]
satisfying the condition \eqref{EQ:bdd_diff} for McDiarmid's inequality for each of the $nk$ random variables $(x_i^j, y_i^j)$. Therefore, for any $t > 0$, by  \eqref{EQ:McDiarmid2} we have 
\[
\mathbb{P}\left(\left|\frac{\rho(k,d)}{2n}\sum_{j=1}^n (\W_2^2(\hat{\mu}^j_k,\hat{\mu}'^j_k)-\E \W_2^2(\hat{\mu}_k,\hat{\mu}'_k))\right|\geq t\right)\leq 2e^{-\frac{2kn t^2}{R^4\rho(k,d)^2}}.
\]
Setting $t = \rho(k,d)R^2\sqrt{\frac{\log (kn)}{kn}}$ then yields 
\[
\mathbb{P}\left(\left|\frac{\rho(k,d)}{2n}\sum_{j=1}^n\left( \W_2^2(\hat{\mu}^j_k,\hat{\mu}'^j_k)-\E \W_2^2(\hat{\mu}^j_k,\hat{\mu}'^j_k)\right)\right|\geq  \rho(k,d)R^2 \sqrt{\frac{\log (kn)}{kn}}\right)\leq \frac{2}{k^{2}n^2}.
\]
Recalling the definition of $\mathrm{Var}_k(\mu)$, the theorem results.
\end{proof}
 \begin{remark}[Interpretation of \cref{thm:varVanilla}]
In words, 
as dimension $d$ and size $k$ increase, we need a smaller number $n$ of independent trials $n$ of $k$-samples to estimate $k$-variance accurately.  Eventually, even choosing $n=1$ suffices. 
\end{remark}

\begin{remark}[Alternative forms for \cref{thm:varVanilla}]
\cref{thm:varVanilla} is written in terms of the number $n$ of sets of $k$ replicates. We can rewrite it in terms of $k$ and $m = kn$, i.e., when a total of $m$ samples are available and one is choosing a $k$ to partition them. We have for $d > 2$ 
\begin{equation*}
\mathbb{P}\left(\left|\widehat{\mathrm{Var}}_k(\mu)-\mathrm{Var}_k(\mu)\right| \geq  \rho(k,d) R^2 \sqrt{\frac{\log m}{m}}\right)\leq \frac{2}{m^2}.
\end{equation*}
This in a sense reverses the tradeoff, with finer divisions of the $m$ available samples (smaller $k$) reducing the overall variance (only slightly for large $d$, however).
\end{remark}

\section{Experiments}\label{sec:experiments}

In this section, we provide some simple experiments demonstrating the behavior of $\Var_k(\cdot)$ and suggesting how it might be used to understand properties of distributions and datasets that are not well-captured by variance alone. 

\subsection{Gaussian Mixtures}\label{sec:gmm}

\begin{figure}
    \centering
    \begin{tabular}{cc}
    \includegraphics[width=.25\linewidth]{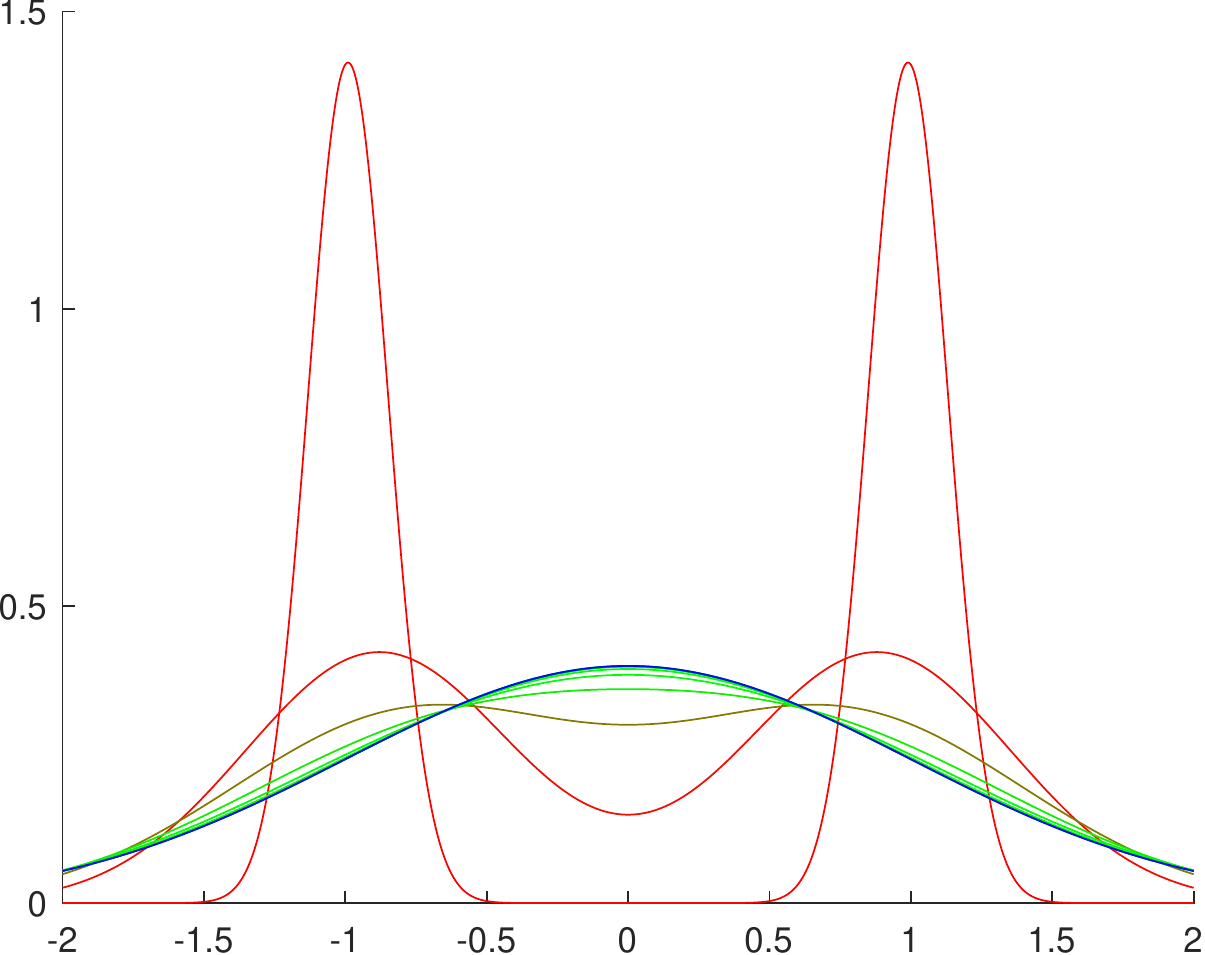} &
    \includegraphics[width=.25\linewidth]{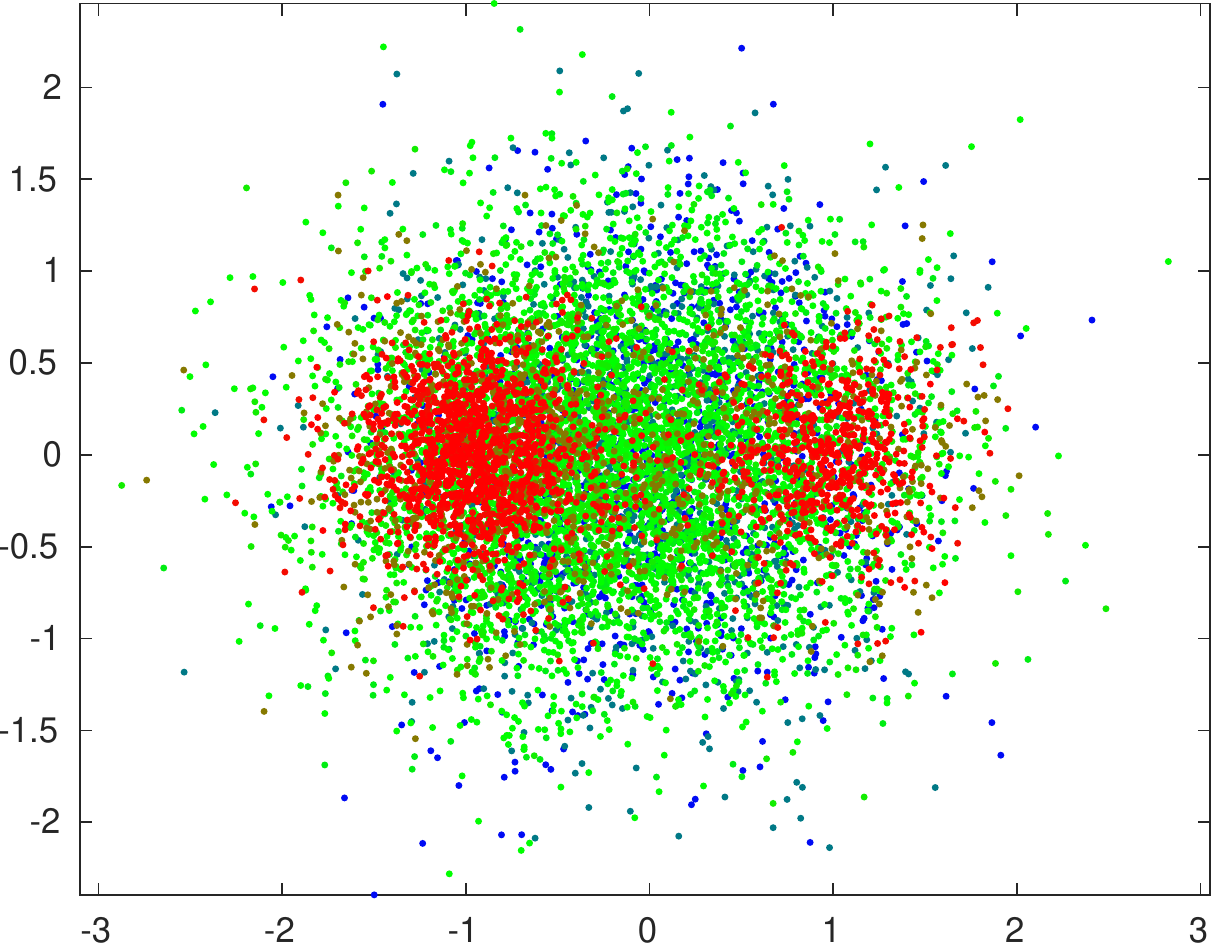}\\
    (a) $d=1$ & (b) $d=2$\\ \emph{\scriptsize Distribution functions} & \emph{\scriptsize Samples}
    \end{tabular}
    \caption{Example distributions for the experiments in \cref{sec:gmm}.  For $d=1$ we show the density functions corresponding to the different colors, and for $d=2$ we show samples drawn from the various measures.}
    \label{fig:samplegmm}
\end{figure}

\begin{figure}
    \centering
    \begin{tabular}{@{}c@{}c@{}c@{}c@{}}
    \includegraphics[width=.25\linewidth]{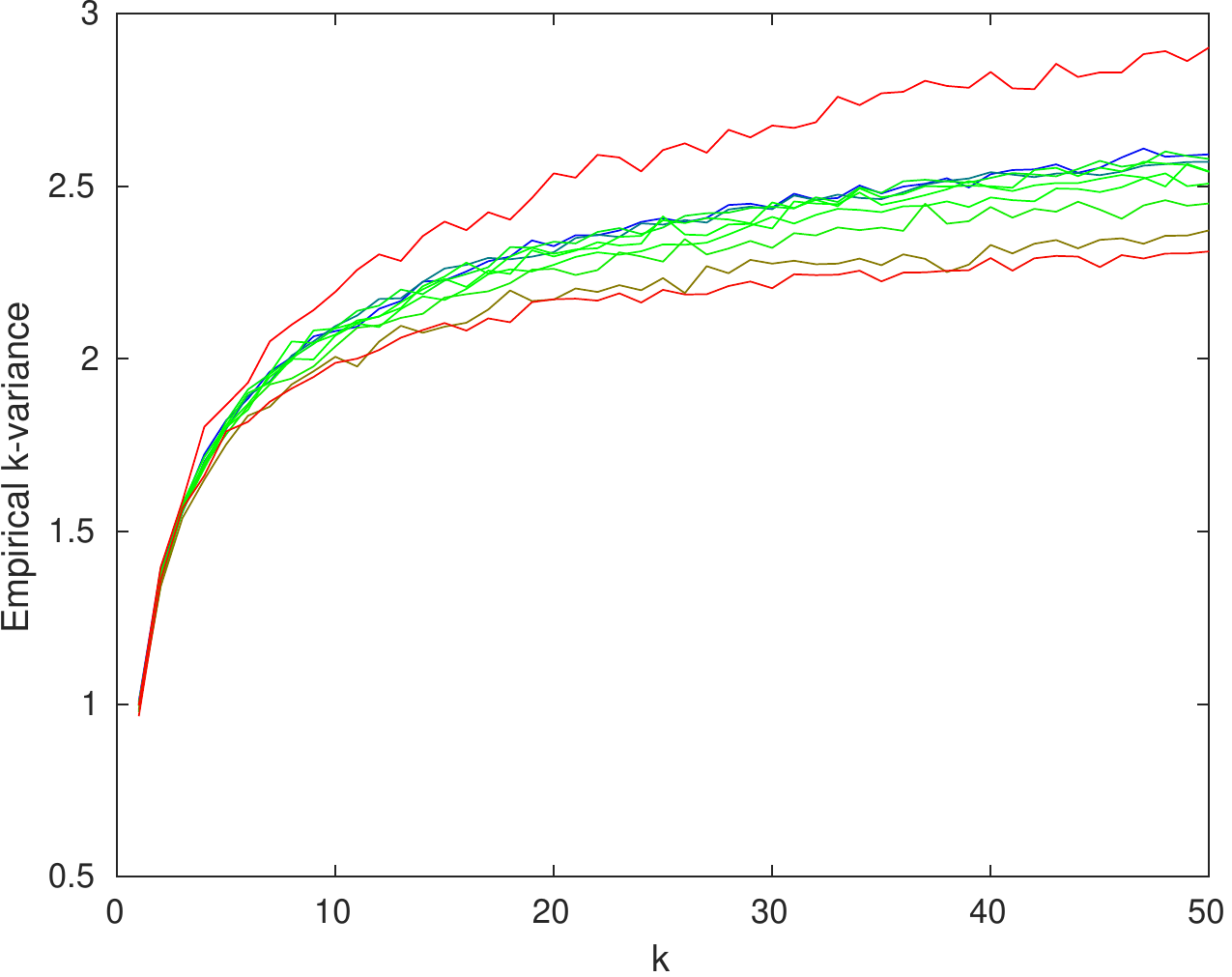}&
    \includegraphics[width=.25\linewidth]{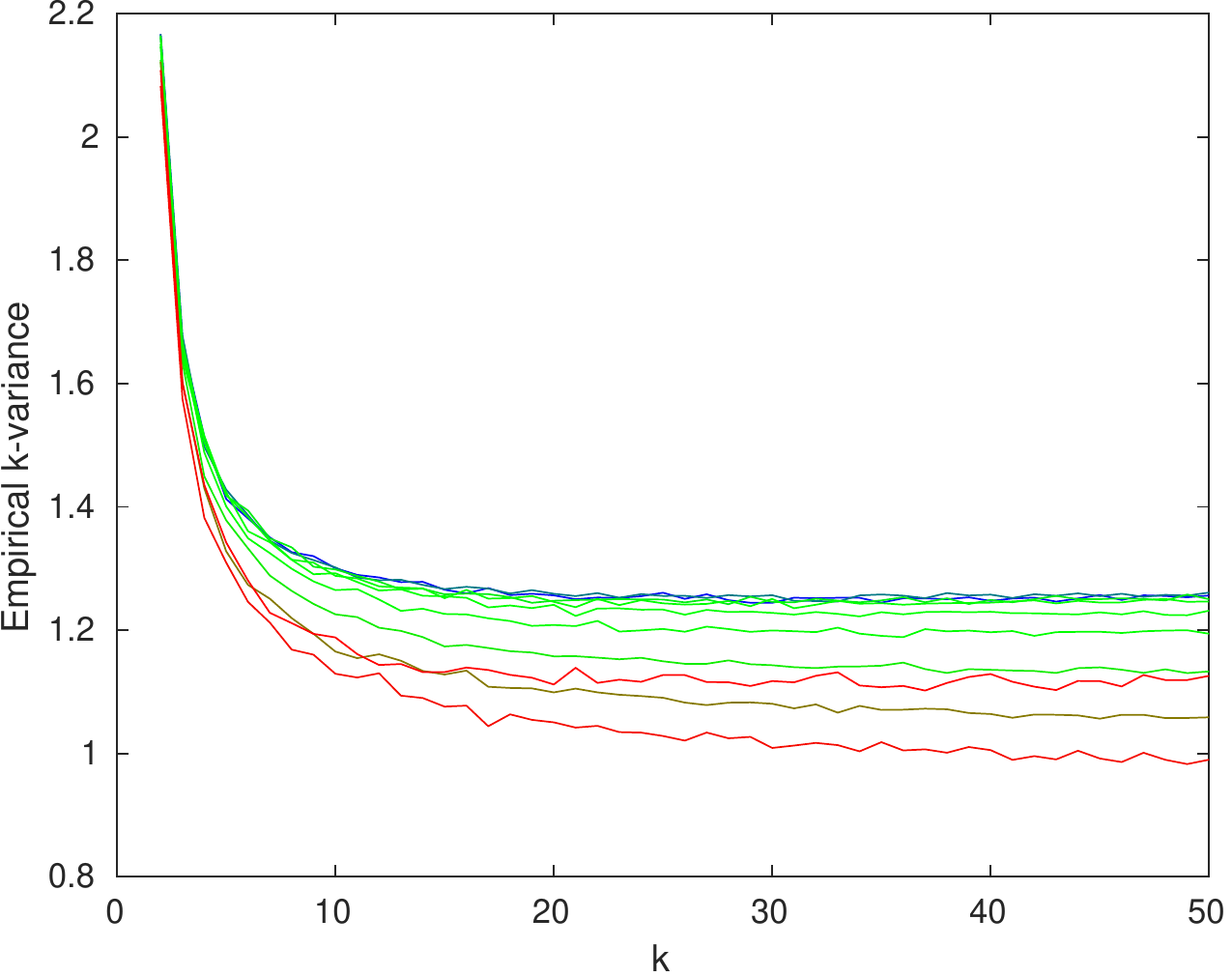}&
    \includegraphics[width=.25\linewidth]{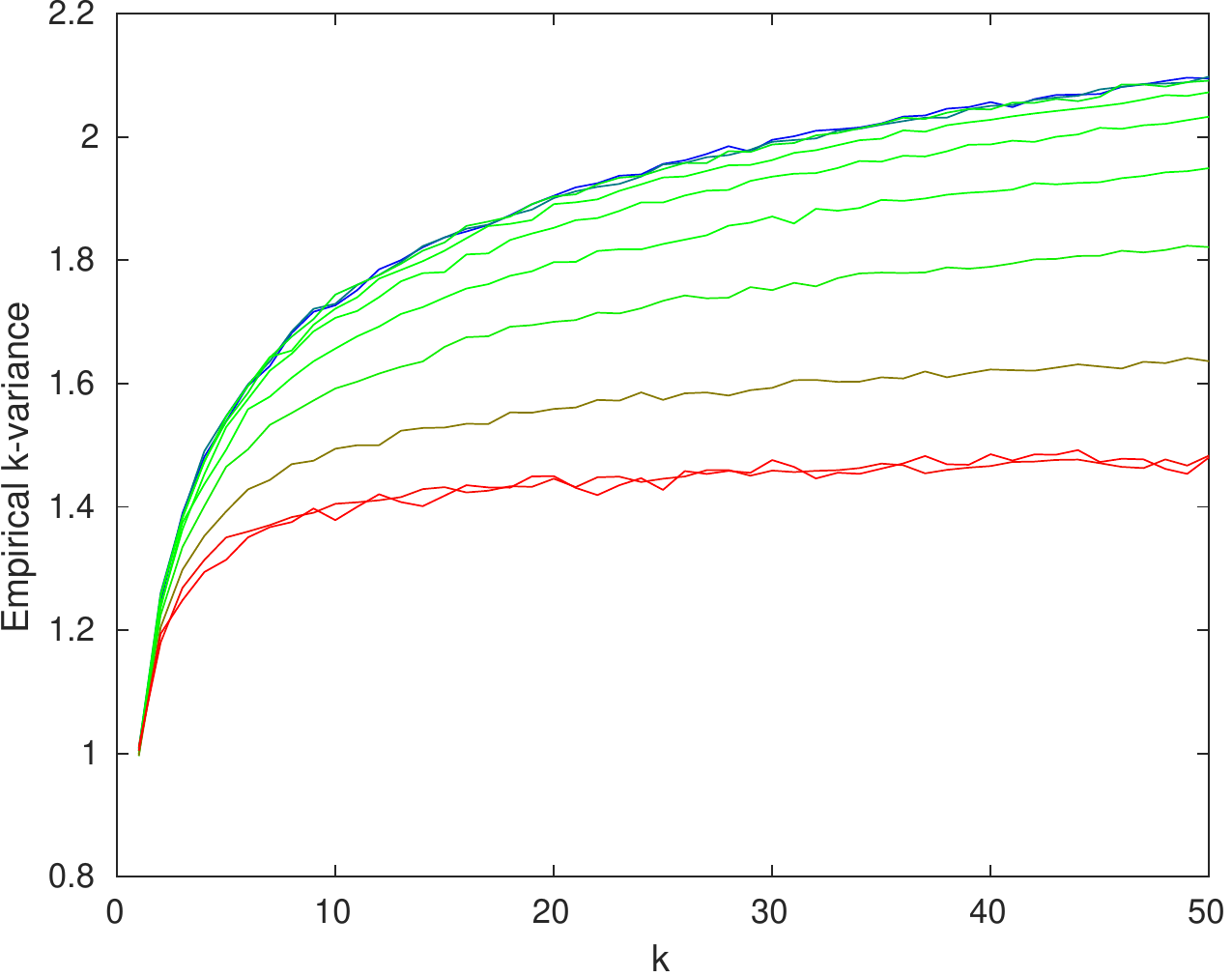}&
    \includegraphics[width=.25\linewidth]{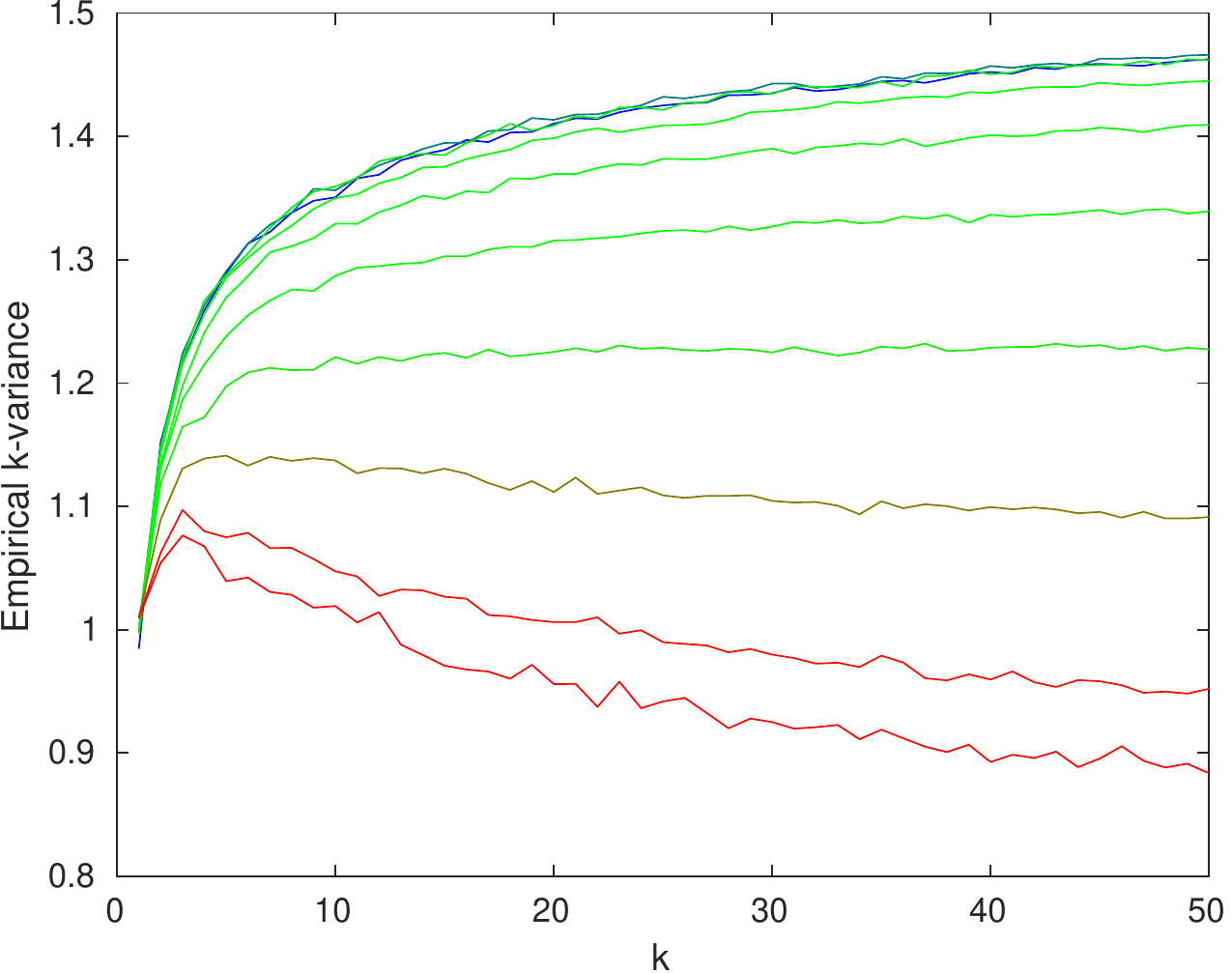}\vspace{-.05in}\\
    $d=1$ & $d=2$ & $d=3$ & $d=4$ \\
    \includegraphics[width=.25\linewidth]{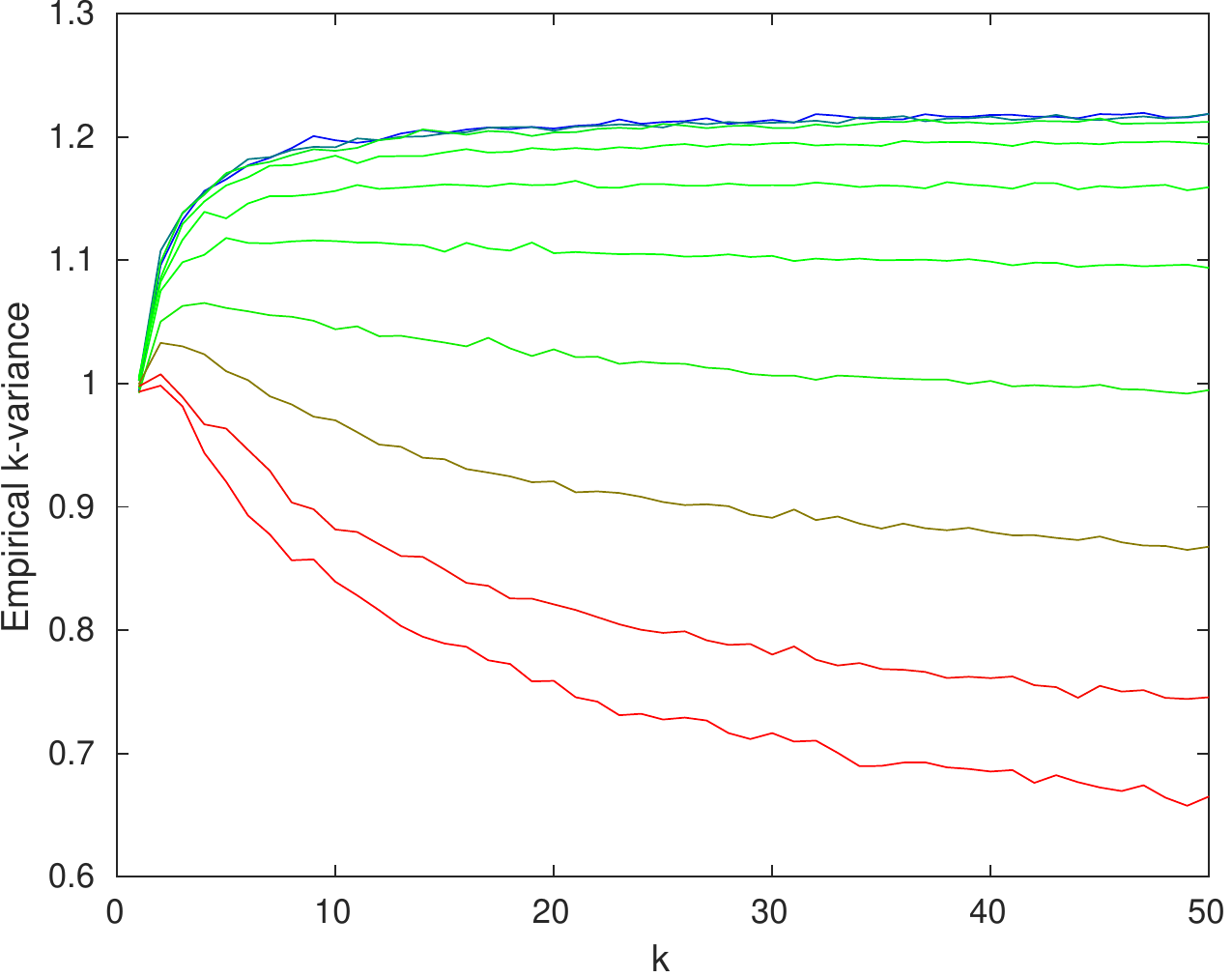}&
    \includegraphics[width=.25\linewidth]{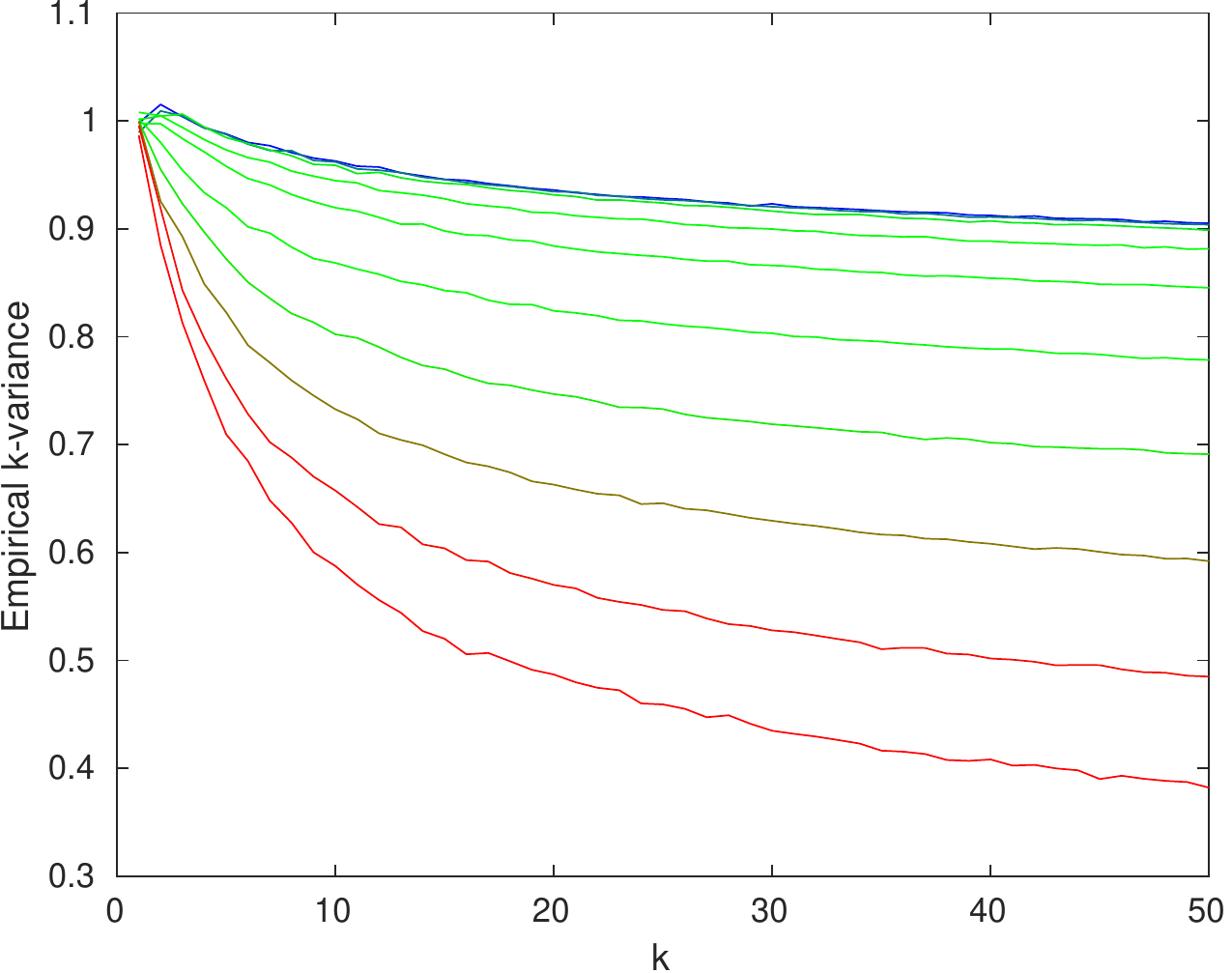}&
    \includegraphics[width=.25\linewidth]{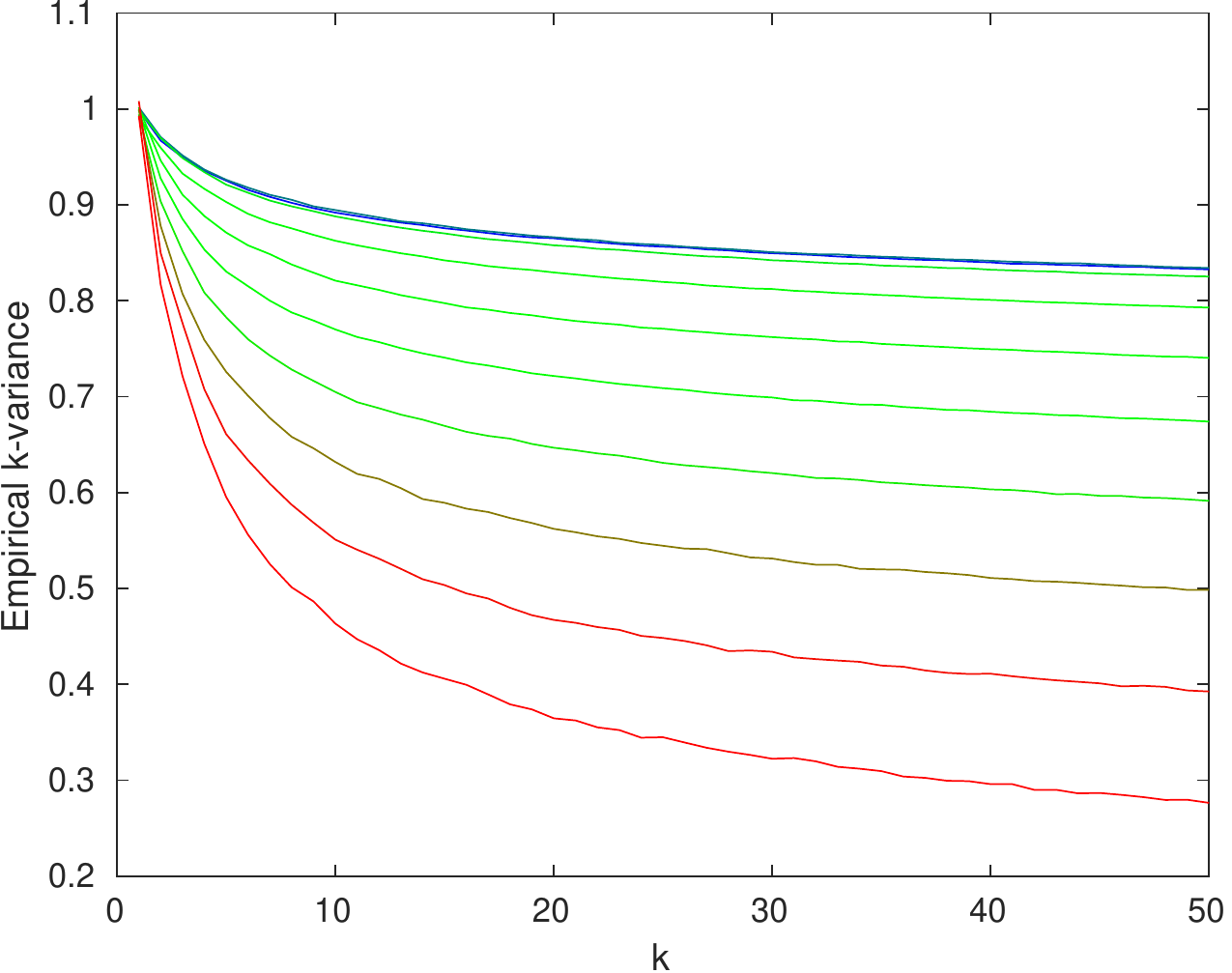}&
    \includegraphics[width=.25\linewidth]{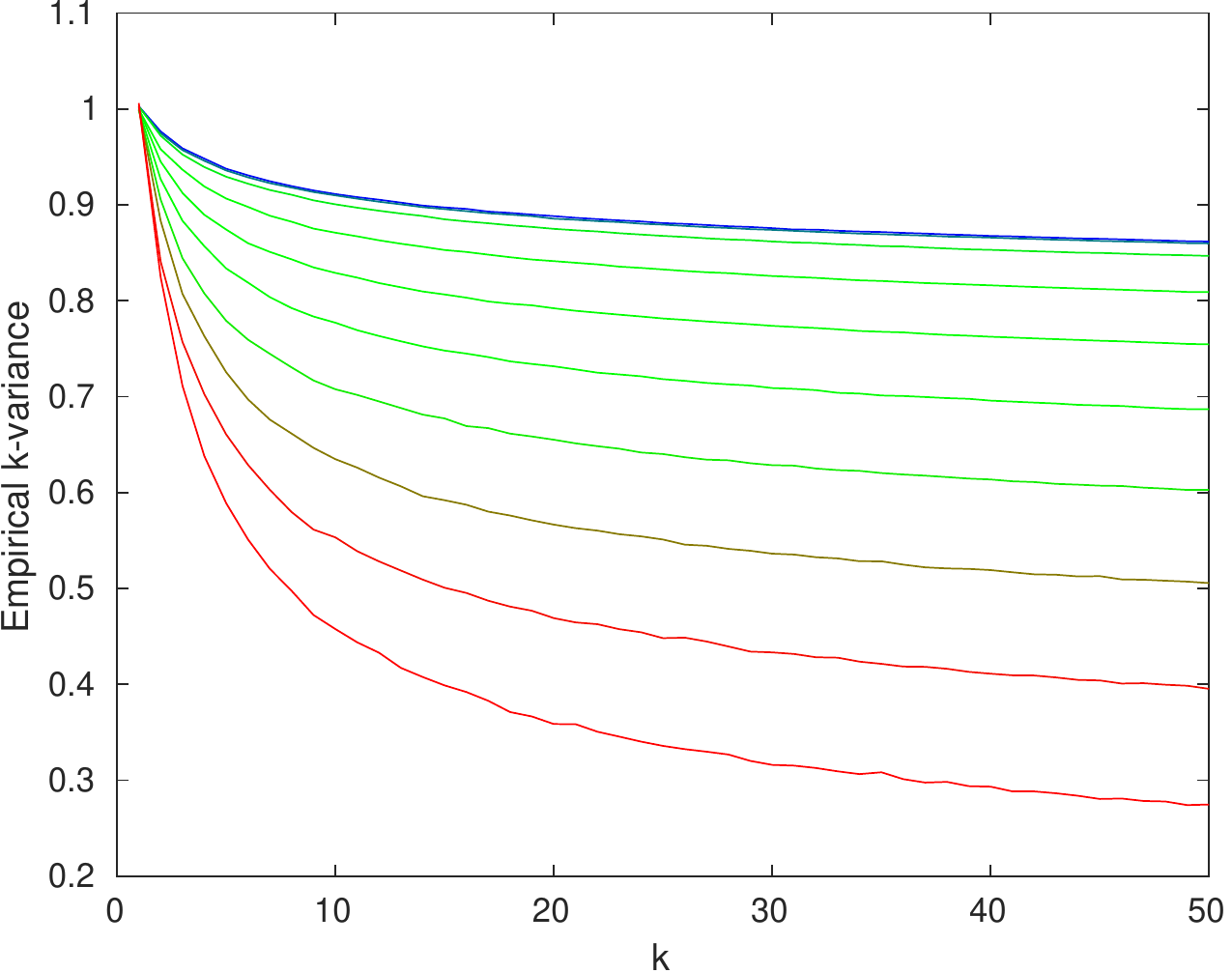}\vspace{-.05in}\\
    $d=5$ & $d=10$ & $d=50$ & $d=100$
    \end{tabular}
    \caption{$k$-variance experiments with Gaussian mixture models (see \cref{sec:gmm}) in increasing dimension.  As predicted, $k$-variance follows similar patterns in $d\geq3$ and is lower for clustered distributions, but for $d\in\{1,2\}$ the behavior is different. Colors range from bimodal mixtures of low-variance Gaussians (red) to unimodal Gaussian measures (blue); see \cref{fig:samplegmm} for examples in $d\in\{1,2\}.$}
    \label{fig:gmm}
\end{figure}

We begin with a synthetic experiment illustrating the behavior of $k$-variance in different dimensions and in the presence of multimodality.  In our experiments, we consider mixtures $\mathcal G_x:=0.5\mathcal N(-x\cdot e_1,\sigma I_{d\times d})+0.5\mathcal N(x\cdot e_1,\sigma I_{d\times d})$ of two isotropic Gaussians, where $e_1\in\R^d$ is the first standard basis vector in $\R^d$.  We choose $\sigma(x)$ so that $\Var_1(\mathcal G_x)=1$; note $\sigma(x)$ decreases as $|x|$ increases, leading to bimodal/approximately clustered distributions. See \cref{fig:samplegmm} for examples in dimensions 1 and 2.

\cref{fig:gmm} shows $k$-variance of $\mathcal G_x$ as a function of $k$ (horizontal axis) and $x$ (color) in different ambient dimensions $d$.  We use the empirical estimator of $k$-variance averaged over 10,000 trials for each point in the plot.  We can make a number of observations based on this experiment:
\begin{itemize}[leftmargin=*]
    \item For $d\geq3$, the $k$-variance is smaller for clustered distributions (red) than unimodal Gaussians (blue) with identical (1-)variance.
    \item The $d\in\{1,2\}$ cases exhibit unique, nonmonotonic behavior. For instance, when $d=1$, $k$-variance is highest for the sharply bimodal distributions (red), then decreases for wide-and-flat distributions (dark red/green), and then increases again for Gaussians (blue).
    \item For larger dimension $d$, the curves look smoother.  This is a byproduct of the results in \cref{sec:variance}, which predict that the empirical estimator of $\Var_k(\cdot)$ has lower variance in high dimension given a fixed number of samples.
\end{itemize}

\subsection{Low-dimensional measures} 

\begin{figure}
    \centering
    \includegraphics[width=.4\linewidth]{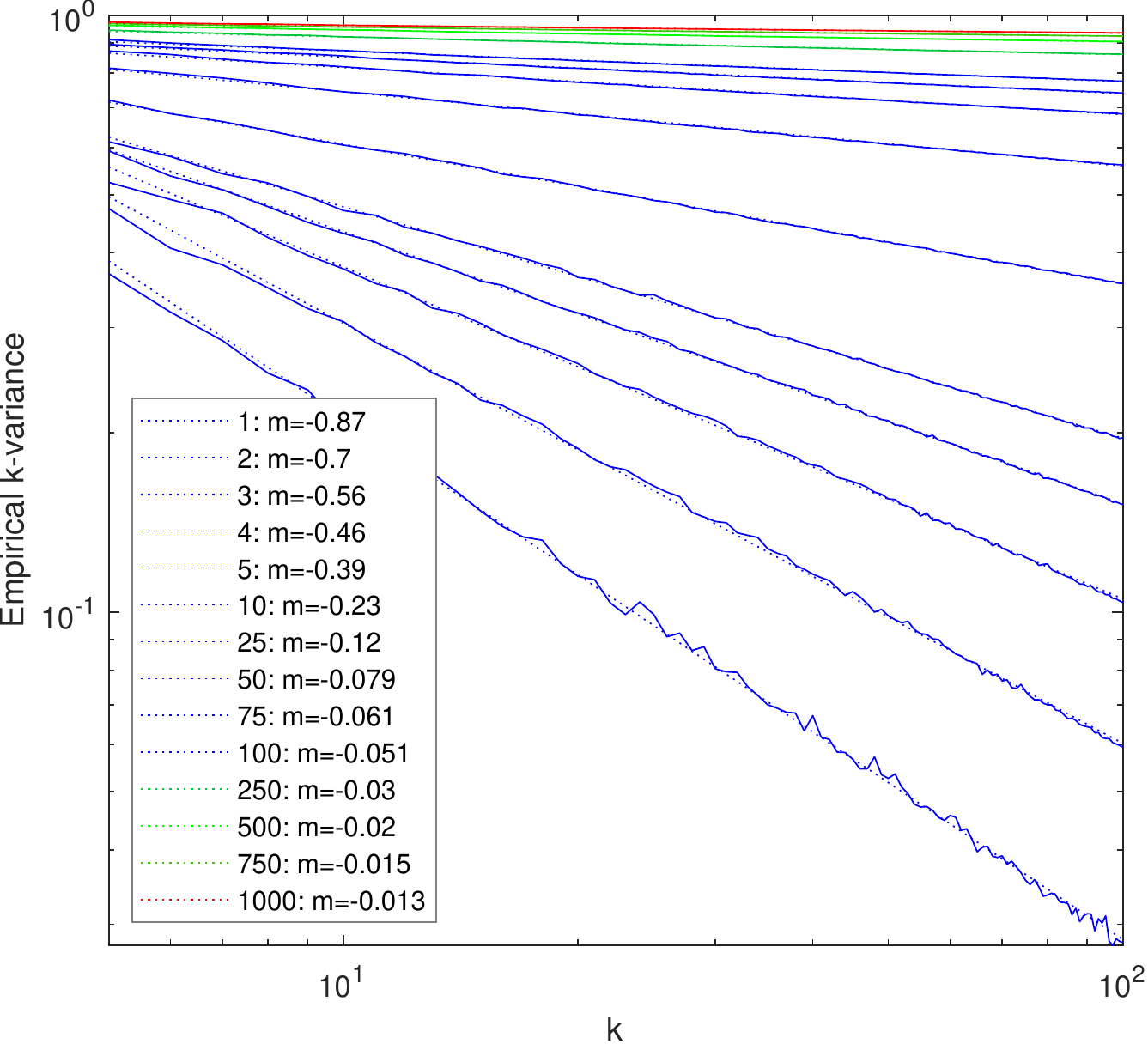}\vspace{-.15in}
    \caption{$k$-variance of measures supported on low-dimensional hyperplanes in $\R^{1000}$.  Each curve corresponds to a different intrinsic dimensionality $d'$ marked in the legend; $m$ is the slope of the best-fit curve in the log-log plot. Note the correlation between $m$ and the intrinsic dimensionality of the measure.}
    \label{fig:cross}
\end{figure}

Now, we consider the case explored in \cref{sec:lowdimensional}, in which our probability measure is embedded in a low-dimensional slice of the ambient space $\R^d$.  When $d$ is sufficiently large, \cref{prop:varklowdim} predicts that the $k$-variance for such a measure will decay to zero at a rate determined by the intrinsic dimensionality of the measure.

As an initial experiment, we consider Gaussian measures in dimension $d=1000$ supported on a $d'$-dimensional hyperplane, where $d'$ varies from 1 to $d$.  Here, we create the $d'$-dimensional measure by creating a Gaussian with covariance $$\Sigma_{d'} = \mathrm{diag}(\underbrace{\nicefrac1{d'},\ldots,\nicefrac1{d'}}_{d'\textrm{ slots}},\underbrace{0,\ldots,0}_{d-d'\textrm{ slots}}).$$ Here, the $\nicefrac1{d'}$ entries ensure that the measure has variance 1.  \cref{fig:cross} plots the $k$-variance of the $d'$-dimensional measures on a logarithmic scale; we use the empirical estimator of $k$-variance averaged over 1,000 trials.  

As predicted by \cref{prop:varklowdim}, the slopes of the fit lines in \cref{fig:cross} cleanly correlated with intrinsic dimensionality.  As $d'$ increases, the lines also become smoother, again a byproduct of the variance bounds in \cref{sec:variance}.  This is a happy coincidence:  We are able to distinguish the slopes of the different lines for large $d'$---even though they are close in value---because we can estimate $\Var_k(\cdot)$ more accurately in this regime.

\begin{figure}
    \centering
    \includegraphics[width=.4\linewidth]{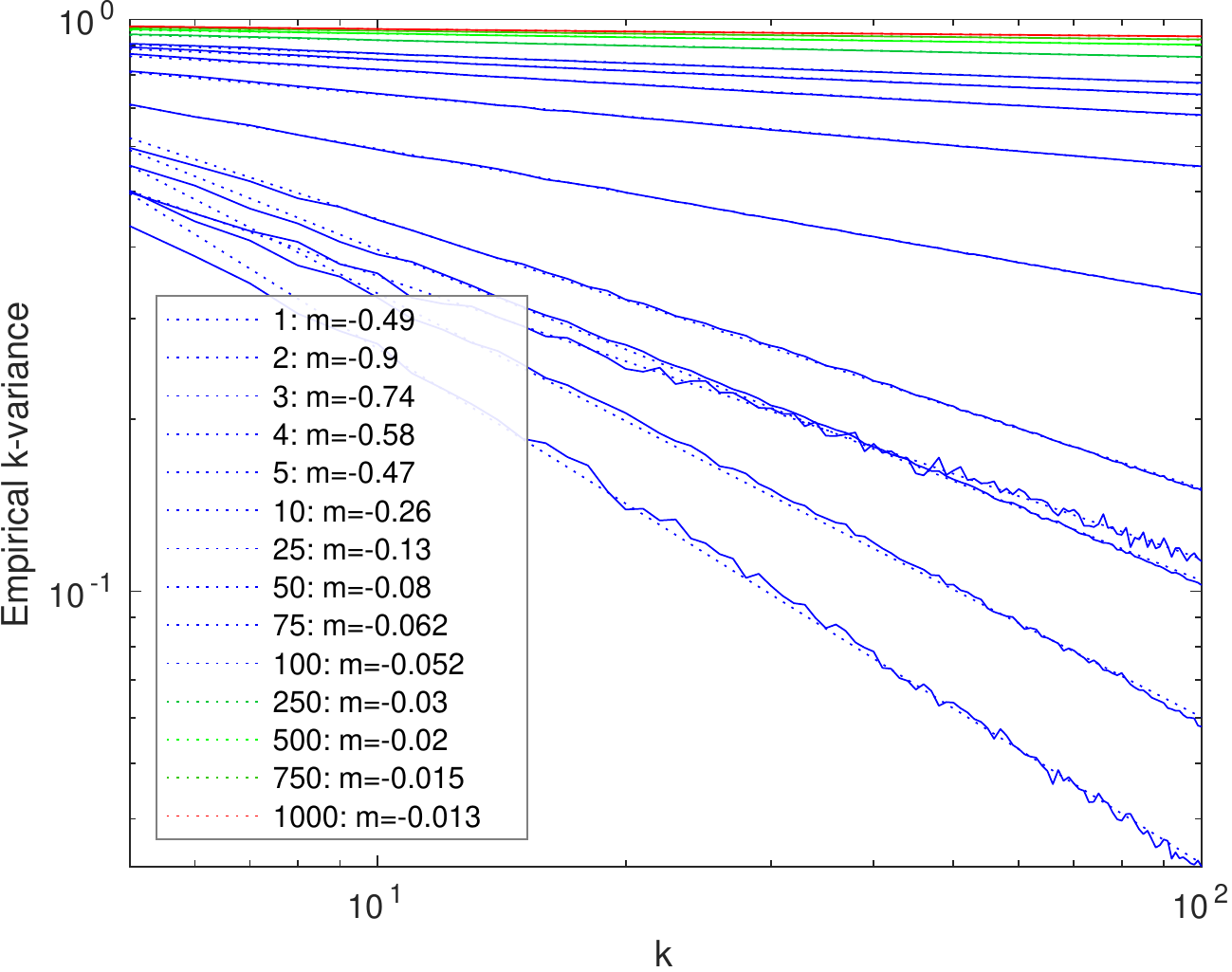}\vspace{-.15in}
    \caption{Similar experiment to \protect\cref{fig:cross}, now with points in the unit sphere $S^{d'-1}$ embedded in $\R^d$.}
    \label{fig:sphere}
\end{figure}

\cref{fig:sphere} shows a similar experiment to \cref{fig:cross}, but now the data lies on the sphere $S^{d'-1}$ embedded in $\R^d$; we sample uniformly from $S^{d'-1}$ by normalizing the samples from the previous experiment to unit length. Once again the trendlines strongly fit the power law we expect to see, but the slopes are now less negative compared to \cref{fig:cross} since the intrinsic dimensionality has decreased by $1$. In particular, note that $d' = 1$ corresponds to a zero-dimensional sphere $S^0$, i.e. two points on the real line. Since $S^0$ is thus a discrete dataset, this explains the approximately $-\nicefrac12$ slope in the log-log plot, corresponding to the $k^{-\nicefrac12}$ decay indicated in \cref{sec:clustered}.

\subsection{Digits}

\begin{figure}
    \centering
    \begin{tabular}{c@{\hspace{.5in}}c}
    \includegraphics[height=2in]{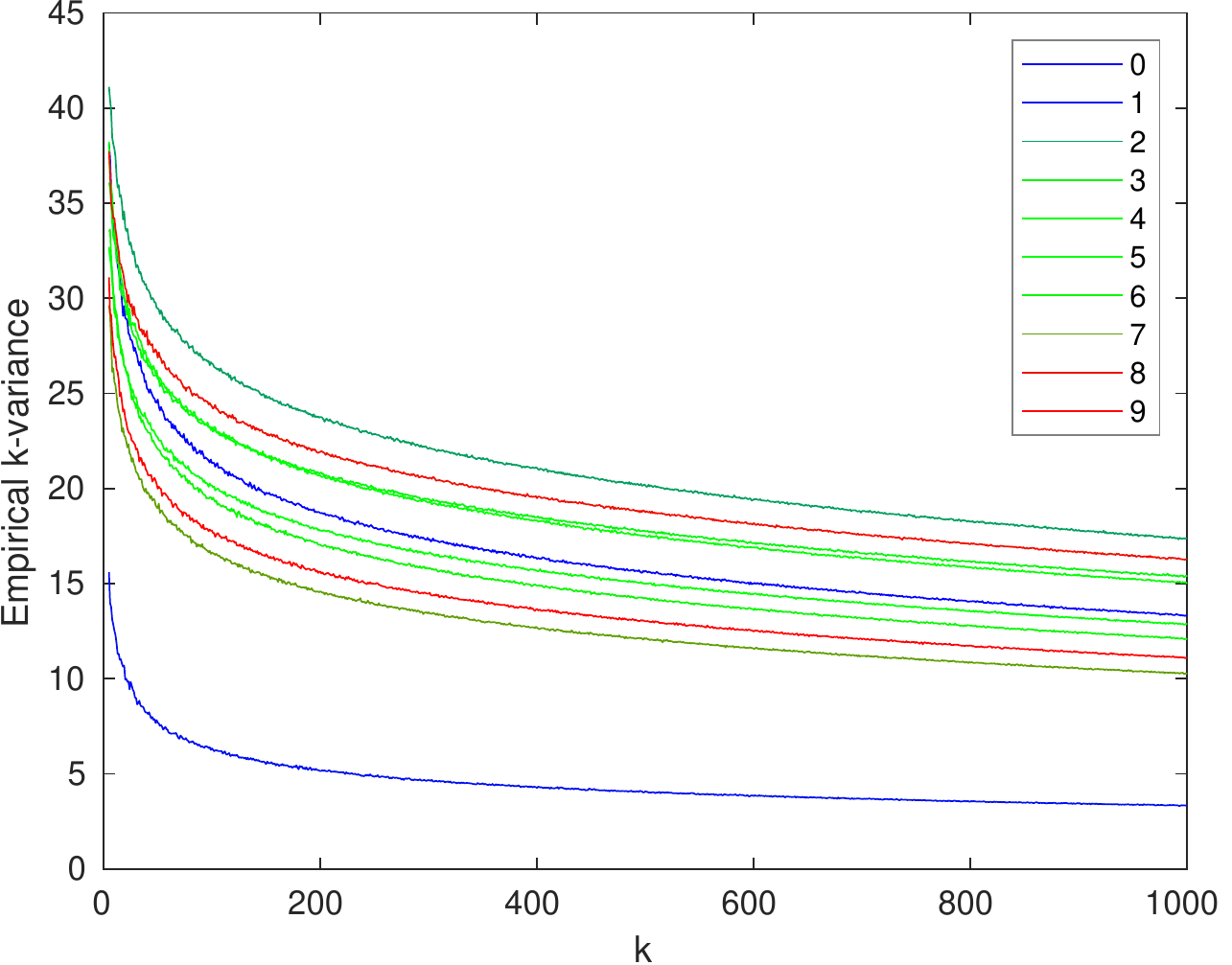} &
    \includegraphics[height=2in]{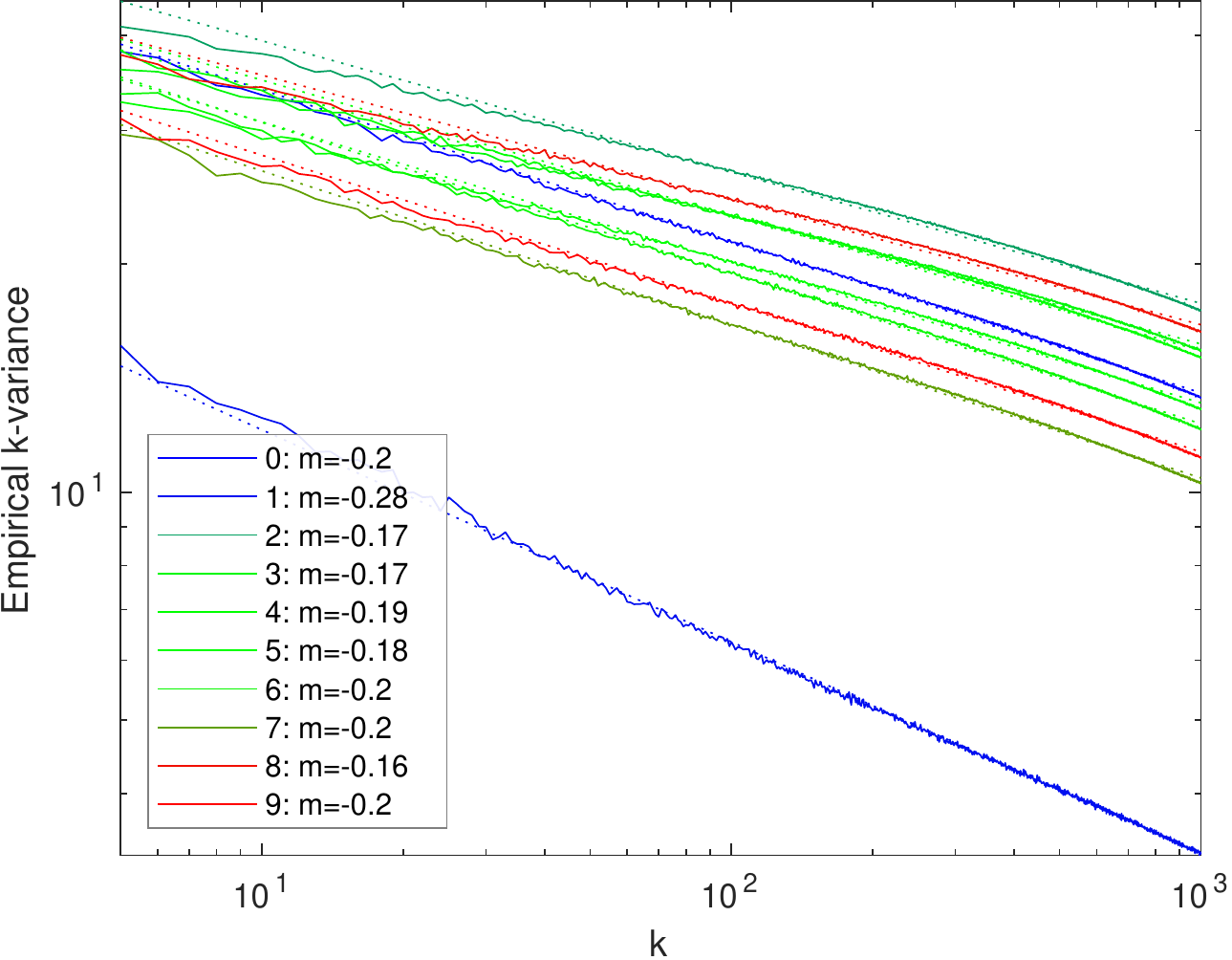} \\
    (a) Linear & (b) Log-log
    \end{tabular}
    \caption{$k$-variance of each MNIST digit from 0 to 9 in (a) linear and (b) log-log scale.  Each digit is considered as a $28^2$-dimensional vector ($d=784$); there are approximately 6,000 images per digit.}
    \label{fig:mnist}
\end{figure}

\cref{fig:mnist} plots approximate $k$-variance for the MNIST dataset of handwritten digits \cite{lecun1998mnist}, separated by digit. We use the stochastic estimator for $k$-variance from \cref{sec:variance}, where sampling from the distribution of handwritten digits is simulated by a bootstrapped strategy of sampling from the dataset with replacement.  Our distributions in this case are over $\R^{784}$, representing $28\times28$ images.  Given the high ambient dimension and the well-documented observation from past work that the MNIST digits roughly lie on low-dimensional submanifolds of $\R^{784}$, we expect $k$-variance to diminish to zero in this experiment. So, the relevant measurement is the rate at which this decay occurs.

Beyond varying amounts of variance between different digits ($k=1$), our experiments also reveal that the digit ``1'' has $k$-variance decaying in $k$ roughly $1.5\!\times$ faster than the other digits.  This provides a quantitative indicator of the observation that there are fewer variations in the way ``1'' is written relative to other digits.  

Less importantly, on the far right of the plots we see decay of the $k$-variance begin to accelerate.  This downward turn occurs roughly at the size of the dataset, because at this scale the bootstrapped estimator becomes less effective:  For extremely large $k$ the dataset looks like a collection of discrete points rather than a smooth distribution over $\R^{784}$.

\section{Conclusion and Future Work}

We can compute $k$-variance easily using a few lines of code, revealing potentially interesting structure hidden in a dataset or probability distribution. Hence, it is a straightforward addition to the data analysis toolkit.  While its properties in $\leq4$ dimensions are somewhat unexpected, beyond this point $k$-variance provides an intuitive means of measuring intra-cluster variance.  Somewhat surprisingly given the ``curse of dimensionality'' associated to optimal transport \cite{weed2019sharp}, we can use fewer data points to estimate $k$-variance of high-dimensional measures, as shown in \cref{sec:variance}.

Beyond its immediate relevance as an analytical tool, $k$-variance motivates a wide variety of challenging research problems moving forward:
\begin{itemize}[leftmargin=*]
    \item Are there nontrivial pairs of measures $\mu,\nu\in\Prob(\R^d)$ with $\Var_k(\mu)=\Var_k(\nu)$ for all $k\geq1$?  Under what conditions can a measure be reconstructed from its mean and sequence of $k$-variance values? 
    \item Beyond the empirical estimator proposed in this paper, are there more efficient or unbiased stochastic estimators for $k$-variance?
    \item Is it possible to generalize $k$-variance to a notion of ``$k$-covariance'' for $d>1$?
    \item Are there analogs of $k$-variance for higher-order moments of a measure?
    \item How do gradient flows of $k$-variance behave?
\end{itemize}

\section*{Acknowledgments}
The authors thank Philippe Rigollet for early feedback and in particular noticing the connection of our work to random bipartite matching and to \cite{bobkov2019one}; Lawrence Stewart for early discussion and experiments; David Wu for early discussions and help deriving combinatorial identities; Mikhail Yurochkin for discussion and feedback; and David Palmer for assistance running some experiments.

\bibliographystyle{siamplain}
\bibliography{references}

\end{document}